\documentclass[12pt, reqno]{amsart}
\usepackage{amssymb,dsfont}
\usepackage{eucal}
\usepackage{amsmath}
\usepackage{amscd}
\usepackage[dvips]{color}
\usepackage{multicol}
\usepackage[all]{xy}           %xypic macro for latex2.09
\usepackage{graphicx}
\usepackage{color}
\usepackage{colordvi}
\usepackage{xspace}
\usepackage{txfonts}
\usepackage{lscape}
\usepackage{tikz} % A REMETTRE

\numberwithin{equation}{section}
\usepackage[shortlabels]{enumitem}
\usepackage{ifpdf}
\ifpdf
 \usepackage[colorlinks,final, hyperindex]{hyperref}
\else
\fi
\usepackage{tikz}
\usepackage[active]{srcltx}

\newtheorem{theorem}{Theorem}[section]
\newtheorem{definition}[theorem]{Definition}
\newtheorem{lemma}[theorem]{Lemma}%[section]
\newtheorem{proposition}[theorem]{Proposition}%[section]
\newtheorem{remark}[theorem]{Remark}%[section]
\newtheorem{corollary}[theorem]{Corollary}%[section]
\newtheorem{example}[theorem]{Example}%[section]
\newtheorem{claim}{Claim}
\numberwithin{equation}{section}

\def\bfi{\mathbf{i}}
\def\bfj{\mathbf{j}}
\def\bfk{\mathbf{k}}
\def\bfl{\mathbf{l}}
\def\bfw{\mathbf{w}}

\def\bfh{\mathbf{h}}

\def\sl{\mathfrak{sl}}

\def\al{\alpha}
\def\be{\beta}

\def\hh{\mathfrak{h}}
\def\gg{\mathfrak{g}}

\def \<{\langle}
\def \>{\rangle}

\def\GG{\mathcal{G}}

\def\UU{\mathcal{U}}
\def\VV{\mathcal{V}}
\def\LL{\mathcal{L}}
\def\HH{\mathcal{H}}

\def\WW{\mathcal{W}}

\newcommand{\C}{\mathbb {C}}

\newcommand{\N}{\mathbb{N}}
\newcommand{\Z}{\mathbb{Z}}

\def\Ind{\mathrm{Ind}}
\def\supp{\mathrm{supp}}
\allowdisplaybreaks

\begin{document}

\title[planar Galilean conformal algebra]{Irreducible modules over the universal central extension of the planar Galilean conformal algebra}
\author{Dongfang Gao}
  \address{D. Gao:  Chern Institute of Mathematics and LPMC, Nankai University, Tianjin 300071, P. R. China, and Institut Camille Jordan, Universit\'{e} Claude Bernard Lyon 1, Lyon, 69622, France.}
  \email{gao@math.univ-lyon1.fr}

\maketitle

\begin{abstract}
In this paper, we study the representation theory of the universal central extension $\GG$ of the infinite-dimensional Galilean conformal algebra, introduced by Bagchi-Gopakumar, in $(2+1)$ dimensional space-time, which was named the planar Galilean conformal algebra by Aizawa. More precisely, we construct a family of Whittaker modules $W_{\psi_{m,n}}$ over $\GG$ while the necessary and sufficient conditions for these modules to be irreducible are given when $m\in\Z_+, n\in\N$ and $m,n$ have the same parity. 
Moreover, the irreducible criteria of the tensor product modules $\Omega(\lambda,\eta,\sigma,0)\otimes R, \Omega (\lambda,\eta,0,\sigma)\otimes R$ over $\GG$ are obtained, 
where $\Omega(\lambda,\eta,\sigma,0), \Omega (\lambda,\eta,0,\sigma)$ are $\UU(\hh)$-free modules of rank one over $\GG$ 
and $R$ is an irreducible restricted module over $\GG$.
Also, the isomorphism classes of these tensor product modules are determined.
\end{abstract}

\raggedbottom
\section{Introduction}
The physicists Bagchi and Gopakumar introduced the infinite-dimensional Galilean conformal algebras in \cite{BG} where their motivation is to construct a systematic non-relativistic limit of the AdS/CFT conjecture introduced by Maldacena in \cite{Ma}.
The AdS/CFT correspondence could have a better understanding by investigating these algebras (see \cite{BGMM, MT}).
Moreover, these algebras are related to many well-known infinite-dimensional Lie algebras, for example, the Virasoro algebra (see \cite{Vir}), the Heisenberg-Virasoro algebra (see \cite{ADKP}), the $W(2,2)$ algebra (see \cite{ZD}), the affine Lie algebras (see \cite{Car, Kac}) and so on.
So the infinite-dimensional Galilean conformal algebras have the important and widespread applications in mathematics and mathematical physics. 
 In this paper, we concern the infinite-dimensional Galilean conformal algebra in $(2+1)$ dimensional space-time, named the planar Galilean conformal algebra by Aizawa in \cite{A}.
This algebra is also the special case of \cite{MT}. 
 We will mainly investigate the representation theory of the universal central extension of the planar Galilean conformal algebra since the centers may be important in physics.

For any infinite-dimensional Lie algebra, 
one of the most important problems of the representation theory is that classifying its all irreducible modules/representations. Of course, it is very difficult. So mathematicians try to classify the certain irreducible modules and construct various new irreducible modules by the different ways.
First, weight modules have been the popular modules for many Lie algebras with the triangular
decompositions. In particular, Harish-Chandra modules (weight modules with finite-dimensional weight spaces) are well studied for many infinite-dimensional Lie algebras, for example, 
the Virasoro algebra (see \cite{LZ, M, S}), the twisted Heisenberg-Virasoro algebra (see \cite{LG, LZ1}), 
the affine Kac-Moody Lie algebras (see \cite{CP, GZ}), the Witt algebra of rank $n$ (see \cite{BF}) and so on.
Some results about weight modules with infinite-dimensional weight spaces were obtained in \cite{BBFK, ChenGZ, GaoZ, LZ2, MZ}.

Recently, two families of non-weight modules attract more attentions from mathematicians,
called Whittaker modules and $\UU(\hh)$-free modules respectively, where $\hh$ is the Cartan subalgebra of the Lie algebra.
 Whittaker modules were introduced by Kostant in \cite{K} where he systematically studied Whittaker modules over the arbitrary finite-dimensional complex semisimple Lie algebra $\mathfrak{l}$. In particular, he proved that Whittaker modules with a fixed regular Whittaker function (Lie homomorphism) on a nilpotent radical are (up to isomorphism) in bijective correspondence with the central characters of $\UU(\mathfrak{l})$.
 Actually, these modules for $sl_2(\C)$ were earlier constructed by Arnal and Pinczon in \cite{AP}. 
So far, Whittaker modules for many other Lie algebras have been investigated (see \cite{ALZ, BO, Chr, GMZ, LZ2, MZ1, Mc}). 
The notation of $\UU(\hh)$-free modules was first introduced by Nilsson in \cite{N} for the simple Lie
algebra $sl_{n+1}(\C)$. At the same time, these modules were introduced in a very different approach in
\cite{TZ}. Later, $\UU(\hh)$-free modules for many infinite-dimensional Lie algebras are determined, for
example, the Witt algebra in \cite{TZ}, the Kac-Moody algebras in \cite{CTZ} and so on. 

For the infinite-dimensional Lie algebra $\GG$,
Verma modules and coadjoint representations were studied in \cite{A},
Harish-Chandra modules can be obtained from \cite{CLW}, 
the certain  Whittaker modules were investigated in \cite{CYY}, restricted modules under the certain conditions were characterized in \cite{CY, GG}, $\UU(\hh)$-free modules of rank one were determined in \cite{CGZ}. 
In this paper, 
we shall construct a family of Whittaker $\GG$-modules including the Whittaker $\GG$-modules in \cite{CYY} while investigating the irreducible criteria of these modules.
Moreover, we determine the irreducibility and isomorphism classes of tensor product modules of the Whittaker modules and the $\UU(\hh)$-free modules over $\GG$.

The paper is organized as follows.
In Section 2, we recall that how the infinite-dimensional Galilean conformal algebras arose.
Then we review the universal central extension of the planar Galilean conformal algebra.
Also, we collect some results about restricted modules and $\UU(\hh)$-free modules over $\GG$ for later use.
Moreover, we construct a family of Whittaker modules $W_{\psi_{m,n}}$ over $\GG$ where $m,n\in\N$.
In Section 3, we determine the necessary and sufficient conditions for the Whittaker module $W_{\psi_{m,n}}$ to be irreducible when $m\in\Z_+, n\in\N$ and $m, n$ have the same parity, see Theorem \ref{main-1}.
When $m$ and $n$ have the different parities, we give a conjecture by computing some examples.
In Section 4, we obtain the irreducible criteria of tensor product modules $\Omega(\lambda,\eta,\sigma,0)\otimes R$ and $\Omega (\lambda,\eta,0,\sigma)\otimes R$
over $\GG$, where $\Omega(\lambda,\eta,\sigma,0), \Omega (\lambda,\eta,0,\sigma)$ are $\UU(\hh)$-free modules of rank one over $\GG$ and  $R$ is a restricted module over $\GG$, see Theorem \ref{tensor product modules}.
Furthermore, we determine the isomorphism classes of these tensor product modules, see Theorem \ref{iso}.
In Section 5, we present some concrete examples of irreducible $\GG$-modules by applying Theorem \ref{tensor product modules}.

Throughout this paper, we denote by $\Z, \Z_+, \N, \C$ and $\C^*$ the sets of integers, positive integers,  non-negative integers, complex numbers and nonzero complex numbers respectively.
All vector spaces and algebras are over $\C$.
We denote by $\UU(\gg)$ the universal enveloping algebra of a Lie algebra $\gg$.

\section{Notations and preliminaries}

In this section, we recall the source of the infinite-dimensional Galilean conformal algebras in order to show that
the infinite-dimensional Galilean conformal algebras are really important in mathematical physics rather than artificial. Moreover, we collect some known results about restricted modules and $\UU(\hh)$-free modules over the universal central extension $\GG$ of the planar Galilean conformal algebra. 
\subsection{From Galilean algebras to infinite-dimensional Galilean conformal algebras}\label{sub2.1}

The Galil-ean algebra $G(d,1)$ in $(d+1)$ dimensional space-time has a basis $\{J_{ij},P_i,B_i,H~|~i,j=1,2,\cdots,d\}$ satisfying the following commutation relations
\begin{equation}\label{G algebra}
\begin{split}
&[J_{ij}, J_{rs}]=\delta_{ir}J_{js}+\delta_{is}J_{rj}+\delta_{jr}J_{si}+\delta_{js}J_{ir},\\
&[H,B_i]=-P_i,\\
&[J_{ij}, B_r]=-(\delta_{jr} B_i-\delta_{ir}B_j),\\
&[J_{ij},P_r]=-(\delta_{jr}P_i-\delta_{ir}P_j),\\
&[J_{ij},H]=[H, P_i]=[B_i,P_j]=[B_i,B_j]=[P_i,P_j]=0,
\end{split}
\end{equation}
where
	\begin{equation} \label{G field}
	J_{ij}=-(x_i\partial_j-x_j\partial_i),\qquad P_i=\partial_i,\qquad
	B_i=t\partial_i,\qquad \qquad H=-\partial_t,
	\end{equation}
	and $x_i,t$ are variables.
By adding the additional generators $\{D,K,K_i~|~i=1,\cdots,d\}$, where
\begin{equation}\label{GCA field}
\begin{split}
&D=-(\sum_{i=1}^dx_i\partial_i+t\partial_t),\\
&K=-(2t\sum_{i=1}^dx_i\partial_i+t^2\partial_t),\\
&K_i=t^2\partial_i.
\end{split}
\end{equation}
we can get the Galilean conformal algebra in $(d+1)$ dimensional space-time spanned by $$\{J_{ij},P_i,B_i,H,D,K,K_i~|~i,j=1,2,\cdots,d\},$$ satisfying the commutation relations \eqref{G algebra} and
\begin{equation*}
\begin{split}
&[J_{ij}, K_r]=-(K_i\delta_{jr}-K_j\delta_{ir}),\quad [K,B_i]=K_i,\quad\;\;\, [K, P_i]=2B_i,\\
&[H,K_i]=-2B_i,\hspace{2.4cm}  [D,K_i]=-K_i,\quad [D,P_i]=P_i,\\
&[D, H]=H,\hspace{3.0cm}  [H,K]=-2D,\;\;\; [D,K]=-K,\\
&[J_{ij}, D]=[J_{ij}, K]=[D, B_i]=[K, K_i]=[K_i, K_j]=[K_i,B_j]=[K_i, P_j]=0.
\end{split}
\end{equation*}

Now, we denote
\begin{equation*}
\begin{split}
&L^{(-1)}=H,\quad L^{(0)}=D,\quad L^{(1)}=K, \quad M_i^{(-1)}=P_i,\quad M_i^{(0)}=B_i,\quad M_i^{(1)}=K_i.
\end{split}
\end{equation*}
Then the Galilean conformal algebra in $(d+1)$ dimensional space-time is spanned by
\[
\{J_{ij},L^{(n)},M_i^{(n)}~|~i,j=1,2,\cdots,d,n=0,\pm1\}
\]
with the commutation relations
\begin{equation*}
\begin{split}
&[L^{(m)},L^{(n)}]=(m-n)L^{(m+n)},\\
&[L^{(m)},M_i^{(n)}]=(m-n)M_i^{(m+n)},\\
&[J_{ij}, M_k^{(m)}]=-(M_i^{(m)}\delta_{jk}-M_j^{(m)}\delta_{ik}),\\
&[J_{ij}, L^{(n)}]=[M_i^{(m)}, M_j^{(n)}]=0,
\end{split}
\end{equation*}
where $m,n=0,\pm1,i,j=1,2,\cdots,d.$
In fact, we may define the vector fields
\begin{equation*}
\begin{split}
&J_{ij}=-(x_i\partial_j-x_j\partial_i),\\
&L^{(n)}=-(n+1)t^n\sum_{i=1}^dx_i\partial_i-t^{n+1}\partial_t,\\
&M_i^{(n)}=t^{n+1}\partial_i,
\end{split}
\end{equation*}
where $n=0,\pm1,i,j=1,2,\cdots,d$.
These are exactly the vector fields in \eqref{G field} and \eqref{GCA field},
so they generate the Galilean conformal algebra.
Furthermore, we have a very natural extension, for arbitrary $n\in\Z, 1\leq i,j\leq d$, define
\begin{equation*}
\begin{split}
&L^{(n)}=-(n+1)t^n\sum_{i=1}^dx_i\partial_i-t^{n+1}\partial_t,\\
&M_i^{(n)}=t^{n+1}\partial_i,\\
&J_{ij}^{(n)}=-t^n(x_i\partial_j-x_j\partial_i).
\end{split}
\end{equation*}
Then we obtain the infinite-dimensional Galilean conformal algebra $\mathrm{GCA}$ in $(d+1)$ dimensional space-time
$$
\mathrm{GCA}=\mathrm{span}\{J_{ij}^{(n)},L^{(n)},M_i^{(n)}~|~n\in\Z,i,j=1,2,\cdots,d\},
$$
satisfying the following commutation relations 
\begin{align}
&[L^{(m)},L^{(n)}]=(m-n)L^{(m+n)},\notag\\
&[J_{ij}^{(m)}, J_{rs}^{(n)}]=\delta_{ir}J_{js}^{(m+n)}+\delta_{is}J_{rj}^{(m+n)}+\delta_{jr}J_{si}^{(m+n)}+\delta_{js}J_{ir}^{(m+n)},\label{IGCA}\\
&[L^{(m)},J_{ij}^{(n)}]=-nJ_{ij}^{(m+n)},\qquad [L^{(m)},M_i^{(n)}]=(m-n)M_i^{(m+n)},\notag\\
&[J_{ij}^{(m)},M_k^{(n)}]=-(\delta_{jk}M_i^{(m+n)}-\delta_{ik}M_j^{(m+n)}),\quad [M_i^{(m)},M_j^{(n)}]=0,\notag
\end{align}
for any $m,n\in\Z, i,j,k=1,2,\cdots,d$.
The readers could see \cite{BG} for the more details about the infinite-dimensional Galilean conformal algebras.

\subsection{Universal central extension of the planar Galilean conformal algebra}\label{universal central extension}

From subsection \ref{sub2.1}, we know that the planar Galilean conformal algebra is spanned by $\{J_{12}^{(n)}, L^{(n)},M_i^{(n)}~|~n\in\Z,i=1,2\}$
satisfying the following commutation relations
\begin{equation*}
\begin{split}
&[L^{(m)},L^{(n)}]=(m-n)L^{(m+n)},\qquad [L^{(m)}, J_{12}^{(n)}]=-nJ_{12}^{(m+n)},\\
&[L^{(m)}, M_1^{(n)}]=(m-n)M_1^{(m+n)},\quad\; [L^{(m)}, M_2^{(n)}]=(m-n)M_2^{(m+n)},\\
&[J_{12}^{(m)}, M_1^{(n)}]=M_2^{(m+n)},\qquad \qquad\;  [J_{12}^{(m)}, M_2^{(n)}]=-M_1^{(m+n)},\\
&[J_{12}^{(m)}, J_{12}^{(n)}]=[M_1^{(m)}, M_1^{(n)}]=[M_2^{(m)}, M_2^{(n)}]=[M_1^{(m)}, M_2^{(n)}]=0,
\end{split}
\end{equation*}
for all $m,n\in\Z$.
We would like to simplify the notations (see \cite{CYY}) for convenience.
Let
$$L_n=-L^{(n)},~ H_n=\sqrt{-1}J_{12}^{(n)},~ I_n=M_1^{(n)}+\sqrt{-1}M_2^{(n)},~ J_n=M_1^{(n)}-\sqrt{-1}M_2^{(n)},\quad \forall n\in\Z.$$
Then it is easy to check that $\{L_n, H_n,I_n,J_n~|~n\in\Z\}$ satisfy the following commutation relations
\begin{equation}\label{planar}
\begin{split}
&[L_m, L_n]=(n-m)L_{m+n},\quad [L_m, H_n]=nH_{m+n},\\
&[L_m, I_n]=(n-m)I_{m+n},\quad\;\, [L_m, J_n]=(n-m)J_{m+n},\\
&[H_m,I_n]=I_{m+n},\qquad\qquad\; [H_m,J_n]=-J_{m+n},\\
&[H_m, H_n]=[I_m, I_n]=[J_m, J_n]=[I_m, J_n]=0,\quad\forall m,n\in\Z.
\end{split}
\end{equation}
Moreover, from \cite{GLP} we see that the universal central extension of the planar Galilean conformal algebra is
spanned by $\{L_n, H_n,I_n,J_n, {\bf c_1,c_2,c_3}~|~n\in\Z\}$ subject to the following nontrivial commutation relations
\begin{equation}\label{extension planar}
\begin{split}
&[L_m, L_n]=(n-m)L_{m+n}+\frac{m^3-m}{12}\delta_{m+n,0}{\bf c_1},\\
&[L_m, H_n]=nH_{m+n}+m^2\delta_{m+n,0}{\bf c_2}, \quad [H_m, H_n]=m\delta_{m+n,0}{\bf c_3},\\
&[L_m, I_n]=(n-m)I_{m+n},\quad\;\, [L_m, J_n]=(n-m)J_{m+n},\\
&[H_m,I_n]=I_{m+n},\quad\qquad [H_m,J_n]=-J_{m+n}, \quad\forall m,n\in\Z.
\end{split}
\end{equation}
Now, we may describe the definition of the universal central extension of the planar Galilean conformal algebra as follows.
\begin{definition}\label{Galilean}
	The {\bf universal central extension $\GG$ of the planar Galilean conformal algebra} is an infinite-dimensional Lie algebra with a basis $\{L_n, H_n,I_n,J_n, {\bf c_1,c_2,c_3}~|~n\in\Z\}$ subject to the nontrivial commutation relations \eqref{extension planar}.
\end{definition}

It is clear that $\GG=\oplus_{i\in\Z}\GG^i$ is a $\Z$-graded Lie algebra, where
\begin{equation*}
\GG^i=\left\{
\begin{aligned}
&\C L_i+\C H_i+\C I_i+\C J_i,\ \ \qquad\qquad\qquad\qquad\quad \text{if}\ i\ne 0,\\
&\C L_0+\C H_0+\C I_0+\C J_0+\C {\bf c_1}+\C {\bf c_2}+\C {\bf c_3},\,\,  \quad \text{if}\ i=0,
\end{aligned}
\right.
\end{equation*}
and $\GG$ has a triangular decomposition $\GG=\GG^+\oplus \GG^0\oplus \GG^-$,
where $$\GG^+=\oplus_{i\in\Z_+}\GG^i,\qquad  \GG^-=\oplus_{i\in\Z_+}\GG^{-i}.$$
Furthermore, $\GG$ has the following several important subalgebras.
\begin{enumerate}
	\item The subalgebra $\VV$ spanned by $\{L_m, {\bf c_1}~|~m\in\Z\}$ is the Virasoro algebra, which was introduced by M. A. Virasoro in order to study string theory (see \cite{Vir}).
	\item The subalgebra $\LL$ spanned by $\{L_m, H_m, {\bf c_1, c_2, c_3}~|~m\in\Z\}$ is the Heisenberg-Virasoro algebra,
	which was introduced by E. Arbarello, C. De Concini, V. G. Kac and C. Procesi in \cite{ADKP}, where the authors established a canonical isomorphism between the second cohomology of the certain Lie algebra
	and the second singular cohomology of the certain moduli space.
	\item The subalgebra $\WW$ spanned by $\{L_m, I_m, {\bf c_1}~|~m\in\Z\}$ or $\{L_m, J_m, {\bf c_1}~|~m\in\Z\}$ is the $W$-algebra $W(2,2)$, which was introduced by Zhang and Dong in order to classify the certain simple vertex operator algebras (see \cite{ZD}). 
\end{enumerate}

For convenience, we define the following subalgebras of $\GG$, 
for any $m,n,d\in\N$, denote
\begin{align*}
  &\GG_d=\sum_{i\in\N}(\C L_i+\C H_i+\C I_{i-d}+\C J_{i-d})+\C {\bf c_1}+\C {\bf c_2}+\C {\bf c_3},\\
  &\GG^{(m,n)}=\sum_{i\in\N}(\C L_{m+i}+\C H_{m+i}+\C I_{n+i}+\C J_{n+i})+\C {\bf c_1}+\C {\bf c_2}+\C {\bf c_3},\\
  &\widetilde{IJ}=\sum_{i\in\Z}(\C I_{i}+\C J_{i}).
\end{align*}
Suppose that $V$ is an irreducible $\GG_d$-module, then we have the induced $\GG$-module
$$\Ind(V)=\UU(\GG)\otimes_{\UU(\GG_d)}V.$$
From Theorem 3.1 in \cite{CY} we get the following result.
\begin{lemma}\label{CY-Lemma}
	Let $d\in\N$ and $V$ be an irreducible $\GG_d$-module.
	Suppose that there exists $k\in\Z$ such that 
	\begin{enumerate}[$(1)$]
		\item $d+k\geq 0$ and $k\notin -2\Z_+$;
		\item  the actions of $I_k$ and $J_k$ on $V$ are injective;
		\item $L_iV=H_iV=I_jV=J_jV=0,\quad \forall i>k+d,j>k$.
	\end{enumerate}
	Then the induced module $\Ind(V)$ is an irreducible $\GG$-module.
\end{lemma}

Recall that for any $m,n\in \N$, let $\psi_{m,n}:\GG^{(m, n)}\rightarrow \C$ be a Whittaker function and $\C w_{\psi_{m,n}}$ be the one-dimensional module over $\GG^{(m, n)}$, defined by $x\cdot w_{\psi_{m,n}}=\psi_{m,n}(x)w_{\psi_{m,n}}$ for $x\in\GG^{(m, n)}$.
Then we have the induced $\GG$-module
$$W_{\psi_{m,n}}=\Ind_{\GG^{(m, n)}}^{\GG} \C w_{\psi_{m,n}},$$
which is called the {\bf Whittaker module} over $\GG$ with respect to $\psi_{m,n}$.

\begin{definition}
Let $\gg$ be a Lie algebra. Suppose that $f$ is an automorphism of $\gg$ and $V$ is a $\gg$-module. 
The following actions
$$x\cdot v=f(x)v,\quad \forall x\in\gg,v\in V,$$
give $V$ a new $\gg$-module structure, denoted by $V'$. Then $V$ and $V'$ are called {\bf equivalent} $\gg$-modules.
\end{definition}

It is easy to see that equivalent modules over the Lie algebra have the same irreducibility.

\begin{definition}
Let $\gg=\oplus_{i\in\Z}{\gg}_{i}$ be a $\Z$-graded Lie algebra.
A $\gg$-module $V$ is called a {\bf restricted module} if for every $v\in V$ there exists $n\in\N$ such that $\gg_iv=0$
for $i>n$. 
\end{definition}
\begin{remark}
One can define restricted modules over $\GG$ since $\GG$ is a $\Z$-graded Lie algebra.
It is clear that the Whittaker modules over $\GG$ are restricted modules.
\end{remark}

We conclude this section by reviewing the following results about $\UU(\hh)$-free modules over $\GG$.

\begin{definition}
\begin{enumerate}[$(1)$]
      \setlength{\itemindent}{-1.6em}\item For any $\lambda\in\C^*,\eta\in\C,(0\ne)\sigma\in\C[X]$, the polynomial algebra $\C[X,Y]$ has a $\GG$-module structure with the following actions
\begin{align}\label{one}
        &L_m(f(X,Y))=\lambda^{m}f(X,Y-m)(Y-mX+m\eta),\notag\\
        &H_m(f(X,Y))=\lambda^mXf(X,Y-m),\\
        &I_m(f(X,Y))=\lambda^m\sigma f(X-1,Y-m),\notag\\
        &J_m(\C[X,Y])={\bf c_1}(\C[X,Y])={\bf c_2}(\C[X,Y])={\bf c_3}(\C[X,Y])=0,\notag
\end{align}
where $m\in\Z, f(X,Y)\in \C[X,Y]$. This module is denoted by  $\Omega(\lambda,\eta,\sigma,0)$.
    \item For any $\lambda\in\C^*, \eta\in\C, (0\ne)\sigma\in\C[X]$, the polynomial algebra $\C[X,Y]$ is endowed with a $\GG$-module structure by the following actions
    \begin{equation}\label{two}
    \begin{split}
        &L_m(f(X,Y))=\lambda^{m}f(X,Y-m)(Y+mX+m\eta),\\
        &H_m(f(X,Y))=\lambda^mXf(X,Y-m),\\
        &I_m(\C[X,Y])={\bf c_1}(\C[X,Y])={\bf c_2}(\C[X,Y])={\bf c_3}(\C[X,Y])=0,\\
        &J_m(f(X,Y))=\lambda^m\sigma f(X+1,Y-m),
    \end{split}
\end{equation}
where $m\in\Z, f(X,Y)\in \C[X,Y]$.
This module is denoted by  $\Omega(\lambda,\eta,0,\sigma)$.
\item For any $\lambda\in\C^*,\delta\in\C[X]$, the polynomial algebra $\C[X,Y]$
can be endowed with a $\GG$-module structure via the following actions
\begin{equation}\label{one}
\begin{split}
&L_m(f(X,Y))=\lambda^{m}f(X,Y-m)(Y+m\delta),\\
&H_m(f(X,Y))=\lambda^mXf(X,Y-m),\\
&I_m(\C[X,Y])=J_m(\C[X,Y])={\bf c_1}(\C[X,Y])={\bf c_2}(\C[X,Y])={\bf c_3}(\C[X,Y])=0,
\end{split}
\end{equation}
for any $m\in\Z, f(X,Y)\in \C[X,Y]$. We denote this module by $\Omega(\lambda,\delta,0,0)$.
\end{enumerate}
\end{definition}
From \cite{CGZ} we see that the above three families of modules exhaust all $\UU(\hh)$-free modules of rank $1$ over $\GG$.
Moreover, by Theorems 4.13, 4.14 and 4.15 in \cite{CGZ} we have the following lemma.
\begin{lemma}\label{U(h)-irr}
	Let $\lambda\in\C^*,\eta\in\C,\sigma,\delta\in\C[X]$ with $\sigma\ne 0$. Then
	\begin{enumerate}[$(1)$]
		\item $\Omega(\lambda,\eta,\sigma,0)$ is an irreducible $\GG$-module if and only if $\sigma$ is a nonzero constant;
		\item $\Omega(\lambda,\eta,0,\sigma)$ is an irreducible $\GG$-module if and only if $\sigma$ is a nonzero constant;
		\item $\Omega(\lambda,\delta,0,0)$ is always reducible as $\GG$-module.
	\end{enumerate}
\end{lemma}

\section{Whittaker modules over $\GG$}\label{sec.3}

In this section, we determine the necessary and sufficient conditions 
for the Whittaker module $W_{\psi_{m,n}}$ over $\GG$ to be irreducible for any $m\in\Z_+, n\in\N$ with $m+n\in 2\N$.
We study some examples when $m\in\Z_+, n\in\N$ with $m+n\in 2\N+1$.
Consequently, we state a conjecture from these examples.

It is easy to see that the module $W_{\psi_{m,n}}$ is always reducible when $m=0$. Therefore we always assume $m\in\Z_+$ in the remaining parts of this section. 
Recall that for any $m\in\Z_+,n\in\N$, suppose $\psi_{m,n}:\GG^{(m, n)}\rightarrow \C$ is a Whittaker function 
and $$W_{\psi_{m,n}}=\Ind_{\GG^{(m, n)}}^{\GG} \C w_{\psi_{m,n}},$$
is  the corresponding Whittaker module over $\GG$ with respect to $\psi_{m,n}$.
Obviously, the following equations hold
\begin{equation}\label{derived-algebra}
\psi_{m,n}(L_{2m+1+j})=\psi_{m,n}(H_{2m+j})=\psi_{m,n}(I_{m+n+j})=\psi_{m,n}(J_{m+n+j})=0,\,\,\forall  j\in\N,
\end{equation}
 since $\psi_{m,n}([\GG^{(m, n)},\GG^{(m, n)}])=0$.

\begin{proposition}\label{m>n-0}
	Let $m\in\Z_+, n\in\N$ with $m\geq n$. Suppose that $\psi_{m,n}$ is given as above. If $\psi_{m,n}(I_{m+n-1})\psi_{m,n}(J_{m+n-1})\ne 0$,
	 then the induced module
	$W'=\Ind_{\GG^{(m, n)}}^{\GG^{(m,0)}} \C w_{\psi_{m,n}}$ is an irreducible module over $\GG^{(m,0)}$.
	Moreover, 
\begin{align}
&L_{m+n+i}(w)=\psi_{m,n}(L_{m+n+i})w,\ \ H_{m+n+i}(w)=\psi_{m,n}(H_{m+n+i})w, \label{m>n-1}\\
&I_{n+i}(w)=\psi_{m,n}(I_{n+i})w,\ \ J_{n+i}(w)=\psi_{m,n}(J_{n+i})w,  \ \ \ \forall w\in W', i\in\N.\label{m>n-2}
\end{align}
\end{proposition}

Before giving the proof of the above proposition, for any $l\in\Z_+$, we define a total order on $\N^l\times\N^l$,
where $\N^l$ denotes the set of all vectors of the form $\mathbf{i}
:=(i_{l-1},\cdots,i_1,i_0)$ with entries in $\N$. Let $\mathbf{0}$ denote the element
$(0,\cdots,0,0)\in\N^l$ and $\epsilon_i$ denote the element
$(0,\cdots,0,1,0,\cdots,0)\in\N^l$ for $0\leq i\leq l-1$, where $1$ is in the
$i$-th position from the right.
For example, $\epsilon_0=(0,\cdots,0,1),~\epsilon_1=(0,\cdots,0,1,0)$.
For any nonzero $\bfi\in\N^l$, denote
\begin{equation*}
\bfw(\bfi)=\sum_{k=0}^{l-1}(l-k)i_k,
\end{equation*}
which is a nonnegative integer.

Let $\succ$ denote the {\bf reverse lexicographic} total order on $\N^l$,
that is, for any $\bfi,\bfj\in \N^l$,
\begin{equation*}
\bfi \succ \bfj\Longleftrightarrow {\rm\ there\ exists\ } 0\leq k\leq l-1 {\rm\ such\ that\ }  i_k>j_k {\rm\ and\ } i_s=j_s,\
\forall 0\leq s<k.
\end{equation*}
Furthermore, we can induce a {\bf principal total order} on $\N^l\times\N^l$, still denoted by $\succ$.
For ${\bf i,\bf j,\bf i',\bf j'}\in \N^l$,
\begin{equation*}
(\bf j,\bf i) \succ(\bf j',\bf i') \Longleftrightarrow
(\bf j,\bf i, \bfw(\bf i)+\bfw(\bf j))
\succ(\bf j',\bf i', \bfw(\bf i')+\bfw(\bf j')).
\end{equation*}

Now we assume $n\in\Z_+$ and take $l=n$. For $\bfi,\bfj\in\N^n$, denote
\begin{equation*}
J^{\bfj}I^{\bfi}=J_{n-1}^{j_{n-1}}\cdots J_1^{j_1}J_0^{j_0}I_{n-1}^{i_{n-1}}\cdots I_1^{i_1}I_0^{i_0}\in\UU(\GG^{(m,0)}).
\end{equation*}
By the PBW Theorem, each element of $W'$ can be uniquely written
in the form
\begin{equation}\label{eqji}
\sum_{\bfj,\bfi\in\N^n} J^{\bfj}I^{\bfi}w_{\bfj,\bfi},
\end{equation}
where all $w_{\bfj,\bfi}\in \C w_{\psi_{m,n}}$ and only finitely many of them
are nonzero. 
For any $w \in W'$ written in the form \eqref{eqji}, we
denote by $\supp(w)$ the set of all $(\bfj,\bfi)\in \N^n\times\N^n$ such that
$w_{\bfj,\bfi}\ne 0$. Furthermore, let $\deg(w)$ denote the maximal
(with respect to the principal total order on $\N^n\times\N^n$) element of $\supp(w)$, called the {\bf degree} of $w$.
Note that we always mean $w\ne 0$ whenever we write $\deg(w)$.

\begin{lemma}\label{m>n-01}
For any $w\in W'\setminus\C w_{\psi_{m,n}}$, suppose that $\deg(w)=({\bf j},{\bf i})$.
The following conclusions hold.
\begin{enumerate}[$(1)$]
\item\label{m<n-01-1} If ${\bf j}\ne {\bf 0}$, denote $r=\min\{k~|~j_k\ne 0\}$, then  $$\deg\bigg((H_{m+n-1-r}-\psi_{m,n}(H_{m+n-1-r}))w\bigg)=({\bf j}-\epsilon_r, {\bf i}).$$
\item \label{m<n-01-2} If ${\bf j}={\bf 0}, {\bf i}\ne {\bf 0}$, denote $s=\max\{k~|~i_k\ne 0\}$, then  
\begin{align*}
&\deg\bigg((H_{m+n-1-s}-\psi_{m,n}(H_{m+n-1-s}))w\bigg)=({{\bf 0},\bf i}-\epsilon_s),\\ 
{\rm or\ \ \ \ }&\deg\bigg((L_{m+n-1-s}-\psi_{m,n}(L_{m+n-1-s}))w\bigg)=({\bf 0}, {\bf i}-\epsilon_s).
\end{align*}
\end{enumerate}
\end{lemma}
\begin{proof}
	We write $w$ in the form \eqref{eqji}. Then we show the proof by comparing degrees.
	
	(1)  We only need to consider the $(\bfk,\bfl)\in \supp(w)$ such that $$(H_{m+n-1-r}-\psi_{m,n}(H_{m+n-1-r}))J^{\bfk}I^{\bfl}w_{\bfk,\bfl}\ne 0.$$
	For any such $(\bfk,\bfl)$, we compute
\begin{align*}
&(H_{m+n-1-r}-\psi_{m,n}(H_{m+n-1-r}))J^{\bfk}I^{\bfl}w_{\bfk,\bfl}\\
=&(H_{m+n-1-r}-\psi_{m,n}(H_{m+n-1-r}))J_{n-1}^{k_{n-1}}\cdots J_1^{k_1}J_0^{k_0}I_{n-1}^{l_{n-1}}\cdots I_1^{l_1}I_0^{l_0}w_{\bfk,\bfl}\\
=&J_{n-1}^{k_{n-1}}\cdots J_{r+1}^{k_{r+1}}[H_{m+n-1-r}, J_r^{k_r}\cdots J_0^{k_0}]I_{n-1}^{l_{n-1}}\cdots I_1^{l_1}I_0^{l_0}w_{\bfk,\bfl}\\
&+J_{n-1}^{k_{n-1}}\cdots J_1^{k_1}J_0^{k_0}I_{n-1}^{l_{n-1}}\cdots I_{r+1}^{k_{r+1}}[H_{m+n-1-r}, I_r^{l_r}\cdots I_0^{l_0}]w_{\bfk,\bfl}.
\end{align*}
Denote $\deg((H_{m+n-1-r}-\psi_{m,n}(H_{m+n-1-r}))J^{\bfk}I^{\bfl}w_{\bfk,\bfl})=(\bfk^*,\bfl^*)$.
The following results hold by simple observations.
\begin{align}
&{\bf k}\succ {\bf k}^* {\rm \ or\ } {\bf k}={\bf k}^*\label{JI-0}.\\
& {\bf l}\succ {\bf l}^*  {\rm \ or\ } {\bf l}={\bf l}^*. \label{JI-1} \\
& \bfw({\bf k}^*)+\bfw({\bf l}^*)\leq \bfw({\bf k})+\bfw({\bf l})-(n-r).\label{JI-2}  \\
&{\rm If\ } k_r\ne 0, then \deg((H_{m+n-1-r}-\psi_{m,n}(H_{m+n-1-r}))J^{\bfk}I^{\bfl}w_{\bfk,\bfl})=({\bf k}-\epsilon_r,{\bf l}).\label{JI-3} 
\end{align}
By \eqref{JI-3} we get $\deg((H_{m+n-1-r}-\psi_{m,n}(H_{m+n-1-r}))J^{\bfj}I^{\bfi}w_{\bfj,\bfi})=({\bf j}-\epsilon_r,{\bf i})$.
Now suppose that $(\bfk,\bfl)\in \supp(w)$ with $(\bfk,\bfl)\ne (\bfj,\bfi)$. 
We have the following statements by \eqref{JI-0}-\eqref{JI-3}.

If $\bfw({\bf k})+\bfw({\bf l})<\bfw({\bf j})+\bfw({\bf i})$, then 
$$\bfw({\bf k}^*)+\bfw({\bf l}^*)\leq \bfw({\bf k})+\bfw({\bf l})-(n-r)<\bfw({\bf j})+\bfw({\bf i})-(n-r)=\bfw({\bf j}-\epsilon_r)+\bfw({\bf i}).$$
Thus $({\bf j}-\epsilon_r,{\bf i})\succ (\bfk^*,\bfl^*)$.

%If $\bfw({\bf k})+\bfw({\bf l})=\bfw({\bf j})+\bfw({\bf i})$ and $\bfw({\bf l})<\bfw({\bf i})$, then
%$$\bfw({\bf k}^*)+\bfw({\bf l}^*)\leq \bfw({\bf k})+\bfw({\bf l})-(n-r)=\bfw({\bf j}-\epsilon_r)+\bfw({\bf i})
%{\rm \ \ and\ \ }\bfw({\bf l}^*)\leq \bfw({\bf l})<\bfw({\bf i}).$$
%So $({\bf j}-\epsilon_r,{\bf i})\succ (\bfk^*,\bfl^*)$.

If $\bfw({\bf k})+\bfw({\bf l})=\bfw({\bf j})+\bfw({\bf i})$ and ${\bf i}\succ {\bf l}$, then
\begin{align*}
\bfw({\bf k}^*)+\bfw({\bf l}^*)\leq \bfw({\bf k})+\bfw({\bf l})-(n-r)=\bfw({\bf j}-\epsilon_r)+\bfw({\bf i}).
\end{align*}
Moreover,  ${\bf i}\succ{\bf l}^*$ since  ${\bf l}\succ{\bf l}^*$ or ${\bf l}={\bf l}^*$.
Therefore $({\bf j}-\epsilon_r,{\bf i})\succ (\bfk^*,\bfl^*)$.

Suppose $\bfw({\bf k})+\bfw({\bf l})=\bfw({\bf j})+\bfw({\bf i}),  {\bf i}={\bf l}$ 
and ${\bf j}\succ {\bf k}$. Then $k_0=k_1=\cdots =k_{r-1}=0$.
If $k_r\ne 0$, then $(\bfk^*,\bfl^*)=({\bfk}-\epsilon_r,\bfl)$.
If $k_r=0$, then $\bfl\succ\bfl^*$.
So we always have $({\bf j}-\epsilon_r,{\bf i})\succ (\bfk^*,\bfl^*)$.

From the above arguments we conclude that $\deg\big((H_{m+n-1-r}-\psi_{m,n}(H_{m+n-1-r}))w\big)=({\bf j}-\epsilon_r, {\bf i})$.

(2) At the moment, $w=I^{\bfi}w_{\bf0,\bfi}+\sum_{\bfk,\bfl\in\N^n} J^{\bfk}I^{\bfl}w_{\bfk,\bfl}$,
where $w_{\bf0,\bfi}\ne 0$ and $(\bf0,\bfi)\succ(\bfk,\bfl)$.
If $(\epsilon_s,\bfi-\epsilon_s)\notin \supp(w)$, then by similar discussions to $(1)$ we can deduce that
$$\deg\big((H_{m+n-1-s}-\psi_{m,n}(H_{m+n-1-s}))w\big)=({{\bf 0},\bf i}-\epsilon_s).$$ 
If $(\epsilon_s,\bfi-\epsilon_s)\in \supp(w)$, 
then $w=I^{\bfi}w_{\bf0,\bfi}+J^{\epsilon_s}I^{{\bf i}-\epsilon_s}w_{\epsilon_s,{\bfi-\epsilon_s}}+{\rm \ other\ terms}$,
where $w_{\bf0,\bfi},w_{\epsilon_s,{\bfi-\epsilon_s}}\ne 0$.
Thus 
\begin{align*}
&(H_{m+n-1-s}-\psi_{m,n}(H_{m+n-1-s}))w\\
=&i_s\psi_{m,n}(I_{m+n-1})I^{{\bfi}-\epsilon_s}w_{\bf0,\bfi}-\psi_{m,n}(J_{m+n-1})I^{{\bf i}-\epsilon_s}w_{\epsilon_s,{\bfi-\epsilon_s}}+{\rm lower-terms},\\
&(L_{m+n-1-s}-\psi_{m,n}(L_{m+n-1-s}))w\\
=&i_s(2s+1-m-n)\psi_{m,n}(I_{m+n-1})I^{{\bfi}-\epsilon_s}w_{\bf0,\bfi}+(2s+1-m-n)\psi_{m,n}(J_{m+n-1})I^{{\bf i}-\epsilon_s}w_{\epsilon_s,{\bfi-\epsilon_s}}+{\rm lower-terms}.
\end{align*}
It is clear that 
\begin{align}\label{2s+1-m-n}
2s+1-m-n\ne 0
\end{align}
 since $m\geq n$ and $0\leq s\leq n-1$.
So $$(H_{m+n-1-s}-\psi_{m,n}(H_{m+n-1-s}))w=(L_{m+n-1-s}-\psi_{m,n}(L_{m+n-1-s}))w=0$$ implies that 
$\psi_{m,n}(I_{m+n-1})=\psi_{m,n}(J_{m+n-1})=0$. Contradiction.
Therefore 
\begin{align*}
&\deg\bigg((H_{m+n-1-s}-\psi_{m,n}(H_{m+n-1-s}))w\bigg)=({{\bf 0},\bf i}-\epsilon_s),\\ 
{\rm or\ \ \ \ }&\deg\bigg((L_{m+n-1-s}-\psi_{m,n}(L_{m+n-1-s}))w\bigg)=({\bf 0}, {\bf i}-\epsilon_s).
\end{align*}
This completes the proof of the lemma.
\end{proof}

Proof of the Proposition \ref{m>n-0}. If $n=0$, then it is trivial that $W'$ is an irreducible $\GG^{(m,0)}$-module. 
If $n\in\Z_+$, then $W'$ is an irreducible $\GG^{(m,0)}$-module directly follows from Lemma \ref{m>n-01}.
The remaining results are clear.

\begin{remark}
In the proof of Lemma \ref{m>n-01} we do not use the $m\geq n$ except Eq.~\eqref{2s+1-m-n}. So the Lemma \ref{m>n-01} still holds when $m<n$ and $m,n$ have the same parity. Therefore the following corollary holds.
\end{remark}

\begin{corollary}\label{m<n-0}
	Let $m,n\in\Z_+$ with $m<n$. Suppose that $m,n$ have the same parity and $\psi_{m,n}$ is given as above. If $\psi_{m,n}(I_{m+n-1})\psi_{m,n}(J_{m+n-1})\ne 0$,
	then the induced module
	$W'=\Ind_{\GG^{(m, n)}}^{\GG^{(m,0)}} \C w_{\psi_{m,n}}$ is an irreducible module over $\GG^{(m,0)}$.
	Furthermore, the Eqs.~\eqref{m>n-1} and \eqref{m>n-2} hold.
\end{corollary}

\begin{proposition}\label{m<n-1}
	Let $m\in\Z_+$ and $V$ be an irreducible $\GG^{(m,0)}$-module. 
	Suppose that there exists an odd number $k\geq m-1$ such that
	\begin{align}
	&L_{k+j}(v)=H_{k+j}(v)=I_{k+j}(v)=J_{k+j}(v)=0, \ \ \ \forall v\in V,\  j\in\Z_+, \label{2m-1}\\
	&I_{k-i}(v)=\alpha_{k-i}v,\ \ J_{k-i}(v)=\beta_{k-i}v, \ \ \ \ \  \forall v\in V,\  0\leq i\leq m-1,\label{2m-2}
	\end{align}
	where $\alpha_{k-i}, \beta_{k-i}\in\C$ for any $0\leq i\leq m-1$ and $\alpha_k\beta_k\ne 0$.
Then the induced module $V_0={\rm Ind}_{\GG^{(m,0)}}^{\GG_0}V$ is an irreducible $\GG_0$-module.
Moreover, 
\begin{enumerate}[$(1)$]
	\item  the actions of $I_k$ and $J_k$ on $V_0$ are injective;
	\item $L_sV_0=H_sV_0=I_sV_0=J_sV_0=0$ for any $s>k$.
\end{enumerate}
\end{proposition}

Before giving the proof of the above proposition,
 we define a total order $\succ$ on $\N^m\times\N^m$ as above.
For ${\bf i,\bf j,\bf i',\bf j'}\in \N^m$,
\begin{equation*}
(\bf j,\bf i) \succ(\bf j',\bf i') \Longleftrightarrow
(\bf j,\bf i, \bfw(\bf i)+\bfw(\bf j))
\succ(\bf j',\bf i', \bfw(\bf i')+\bfw(\bf j')).
\end{equation*}
Let $\bfh,\bfl\in\N^m$. Denote
\begin{equation*}
H^{\bfh}L^{\bfl}=H_{m-1}^{h_{m-1}}\cdots H_1^{h_1}H_0^{h_0}L_{m-1}^{l_{m-1}}\cdots L_1^{l_1}L_0^{l_0}\in\UU(\GG_0).
\end{equation*}
By the PBW Theorem, each element of $V_0$ can be uniquely written
in the form
\begin{equation}\label{eqhl}
\sum_{\bfh,\bfl\in\N^m} H^{\bfh}L^{\bfl}v_{\bfh,\bfl},
\end{equation}
where all $v_{\bfh,\bfl}\in V$ and only finitely many of them
are nonzero. 
For any $v \in V_0$ written in the form \eqref{eqhl}, we
denote by $\supp(v)$ the set of all $(\bfh,\bfl)\in \N^m\times\N^m$ such that
$v_{\bfh,\bfl}\ne 0$. Furthermore, let $\deg(v)$ denote the maximal
(with respect to the principal total order on $\N^m\times\N^m$) element of $\supp(v)$, called the {\bf degree} of $v$.

\begin{lemma}\label{m<n-11}
	For any $v\in V_0\setminus V$, suppose that $\deg(v)=({\bf h},{\bf l})$.
	The following statements hold.
	\begin{enumerate}[$(1)$]
		\item \label{m<n-11-1}If ${\bf h}\ne {\bf 0}$, denote $r=\min\{i~|~h_i\ne 0\}$, 
		then  $$\deg\bigg((I_{k-r}-\alpha_{k-r})v\bigg)=({\bf h}-\epsilon_r, {\bf l}).$$
		\item \label{m<n-11-2}If ${\bf h}={\bf 0}, {\bf l}\ne {\bf 0}$, denote $s=\max\{i~|~l_i\ne 0\}$, then  
		\begin{align*}
		\deg\bigg((I_{k-s}-\alpha_{k-s})v\bigg)=({{\bf 0},\bf l}-\epsilon_s),
		{\rm\ \  or\ \ \ \ }&\deg\bigg((J_{k-s}-\beta_{k-s})v\bigg)=({\bf 0}, {\bf l}-\epsilon_s).
		\end{align*}
	\end{enumerate}
\end{lemma}
\begin{proof}
	The item \ref{m<n-11-1} is similar to item \ref{m<n-01-1} of Lemma \ref{m>n-01}.  
	We only prove the item \ref{m<n-11-2} which involves that $k$ is an odd number.
	In the case, we can write  $$v=L^{\bfl}v_{\bf0,\bfl}+\sum_{\bfh,\bfk\in\N^m} H^{\bfh}L^{\bfk}v_{\bfh,\bfk},$$
	where $v_{\bf0,\bfl}\ne 0$ and $(\bf0,\bfl)\succ(\bfh,\bfk)$.
	If $(\epsilon_s,\bfl-\epsilon_s)\notin \supp(v)$, then we can deduce that
	$$(I_{k-s}-\alpha_{k-s})v
	=(2s-k)\al_kl_sL^{{\bfl}-\epsilon_s}v_{\bf0,\bfl}+{\rm lower-terms}.$$ 
	Since $k$ is an odd number, thus 
	\begin{align}\label{2s-k}
	2s-k\ne 0.
	\end{align}
	 So $$\deg\bigg((I_{k-s}-\alpha_{k-s})v\bigg)=({{\bf 0},\bf l}-\epsilon_s).$$
	If $(\epsilon_s,\bfl-\epsilon_s)\in \supp(v)$, 
	then $v=L^{\bfl}v_{\bf0,\bfl}+H^{\epsilon_s}L^{{\bf l}-\epsilon_s}v_{\epsilon_s,{\bfl-\epsilon_s}}+{\rm \ other\ terms}$,
	where $v_{\bf0,\bfl},v_{\epsilon_s,{\bfl-\epsilon_s}}\ne 0$.
So 
	\begin{align*}
	&(I_{k-s}-\alpha_{k-s})v
	=(2s-k)\al_kl_sL^{{\bfl}-\epsilon_s}v_{\bf0,\bfl}-\al_kL^{{\bf l}-\epsilon_s}v_{\epsilon_s,{\bfl-\epsilon_s}}+{\rm lower-terms},\\
	&(J_{k-s}-\beta_{k-s})v
	=(2s-k)\beta_kl_sL^{{\bfl}-\epsilon_s}v_{\bf0,\bfl}+\beta_kL^{{\bf l}-\epsilon_s}v_{\epsilon_s,{\bfl-\epsilon_s}}+{\rm lower-terms},
	\end{align*}
	which indicate that $(I_{k-s}-\alpha_{k-s})v=(J_{k-s}-\beta_{k-s})v=0$ is impossible.
	Therefore 
\begin{align*}
\deg\bigg((I_{k-s}-\alpha_{k-s})v\bigg)=({{\bf 0},\bf l}-\epsilon_s),
{\rm\ \  or\ \ \ \ }&\deg\bigg((J_{k-s}-\beta_{k-s})v\bigg)=({\bf 0}, {\bf l}-\epsilon_s).
\end{align*}
	This completes the proof of the lemma.
\end{proof}

Now it is clear that Proposition \ref{m<n-1} holds from Lemma \ref{m<n-11}.

\begin{remark}
	In the proof of Lemma \ref{m<n-11} we do not use that $k$ is an odd number except Eq.~\eqref{2s-k}. So the Lemma \ref{m<n-11} still holds when $k\geq 2m$. Therefore we have the following corollary.
\end{remark}

\begin{corollary}\label{m<n-12}
	Let $m\in\Z_+$ and $V$ be an irreducible $\GG^{(m,0)}$-module. 
	Suppose that there exists $k\geq 2m$ such that Eqs.~\eqref{2m-1} and \eqref{2m-2} hold.
	Then the induced module $V_0={\rm Ind}_{\GG^{(m,0)}}^{\GG_0}V$ is an irreducible $\GG_0$-module.
	Moreover, 
	\begin{enumerate}[$(1)$]
		\item  the actions of $I_k$ and $J_k$ on $V_0$ are injective;
		\item $L_sV_0=H_sV_0=I_sV_0=J_sV_0=0$ for any $s>k$.
	\end{enumerate}
\end{corollary}

Now we state the main results of this section.

\begin{theorem}\label{main-1}
	Let $m\in\Z_+, n\in\N$ and $m,n$ have the same parity.
	Suppose that $\psi_{m,n}:\GG^{(m, n)}\rightarrow \C$ is a Whittaker function.
	Then the corresponding Whittaker $\GG$-module $W_{\psi_{m,n}}=\Ind_{\GG^{(m, n)}}^{\GG} \C w_{\psi_{m,n}}$
	is irreducible if and only if $\psi_{m,n}(I_{m+n-1})\psi_{m,n}(J_{m+n-1})\ne 0$.
	In particular, the result agrees with \cite{CYY} when $m=n=1$.
\end{theorem}
\begin{proof}
	$(\Leftarrow)$ 
	It is clear that $$W_{\psi_{m,n}}=\Ind_{\GG^{(m, n)}}^{\GG} \C w_{\psi_{m,n}}\cong\Ind_{\GG_0}^{\GG}\Ind_{\GG^{(m, 0)}}^{\GG_0}\Ind_{\GG^{(m, n)}}^{\GG^{(m,0)}} \C w_{\psi_{m,n}}.$$
Then we consider the following two cases.

Case 1: $m<n$.

Then by Corollary \ref{m<n-0} we see that the $\Ind_{\GG^{(m, n)}}^{\GG^{(m,0)}} \C w_{\psi_{m,n}}$ is an irreducible 
$\GG^{(m, 0)}$-module satisfying the conditions of Corollary \ref{m<n-12} (with $k=m+n-1$).
 So $\Ind_{\GG^{(m, 0)}}^{\GG_0}\Ind_{\GG^{(m, n)}}^{\GG^{(m,0)}} \C w_{\psi_{m,n}}$ is an irreducible $\GG_0$-module 
 and  satisfies the conditions of Lemma \ref{CY-Lemma} (with $d=0, k=m+n-1$). Therefore $\Ind_{\GG_0}^{\GG}\Ind_{\GG^{(m, 0)}}^{\GG_0}\Ind_{\GG^{(m, n)}}^{\GG^{(m,0)}} \C w_{\psi_{m,n}}$ is an irreducible $\GG$-module, that is, $W_{\psi_{m,n}}$ is an irreducible $\GG$-module.

Case 2: $m\geq n$.

Then we denote $\al_p=\psi_{m,n}(I_p), \beta_p=\psi_{m,n}(J_p)$ for any $2n-1\leq p\leq m+n-1$. (If $n=0$, we denote $\al_{-1}=\beta_{-1}=0$.) We define the following four matrices:
\[
A=\begin{pmatrix}
(m+n+1)\al_{m+n-1}&(m+n+2)\al_{m+n-2}&\cdots&(2m+1)\al_{2n-1}\\            
0&(m+n+3)\al_{m+n-1}&\cdots&(2m+2)\al_{2n}\\
0&0&\cdots&\cdots\\
0&0&\cdots&(3m-n+1)\al_{m+n-1}
\end{pmatrix}_{(m-n+1)\times (m-n+1)},\\
\]
\[
B=\begin{pmatrix}
(m+n+1)\be_{m+n-1}&(m+n+2)\be_{m+n-2}&\cdots&(2m+1)\be_{2n-1}\\            
0&(m+n+3)\be_{m+n-1}&\cdots&(2m+2)\be_{2n}\\
0&0&\cdots&\cdots\\
0&0&\cdots&(3m-n+1)\be_{m+n-1}
\end{pmatrix}_{(m-n+1)\times (m-n+1)},
\]
\[
C=\begin{pmatrix}
-\al_{m+n-1}&-\al_{m+n-2}&\cdots&-\al_{2n-1}\\            
0&-\al_{m+n-1}&\cdots&-\al_{2n}\\
0&0&\cdots&\cdots\\
0&0&\cdots&-\al_{m+n-1}
\end{pmatrix}_{(m-n+1)\times (m-n+1)},
\]
\[
D=\begin{pmatrix}
\be_{m+n-1}&\be_{m+n-2}&\cdots&\be_{2n-1}\\            
0&\be_{m+n-1}&\cdots&\be_{2n}\\
0&0&\cdots&\cdots\\
0&0&\cdots&\be_{m+n-1}
\end{pmatrix}_{(m-n+1)\times (m-n+1)}.
\]
By the simple observations we see that the matrix 
$\begin{pmatrix}
A&B\\            
C&D
\end{pmatrix}$
is invertible since $\al_{m+n-1}\beta_{m+n-1}\ne 0$.
Set $$X=(\psi_{m,n}(L_{m+n}),\psi_{m,n}(L_{m+n+1}),\cdots,\psi_{m,n}(L_{2m}))^T,\ Y=(\psi_{m,n}(H_{m+n}),\psi_{m,n}(H_{m+n+1}),\cdots,\psi_{m,n}(H_{2m}))^T.$$
Then there exist $a_0,a_{-1},\cdots,a_{-(m-n)},b_0,b_{-1},\cdots, b_{-(m-n)}\in\C$ such that
\begin{align}\label{matrix}
\begin{pmatrix}
X\\            
Y
\end{pmatrix}=\begin{pmatrix}
A&B\\            
C&D
\end{pmatrix}
\begin{pmatrix}
X_1\\            
Y_1
\end{pmatrix},
\end{align}
where $$X_1=(a_0,a_{-1},\cdots,a_{-(m-n)})^T,\ Y=(b_0,b_{-1},\cdots, b_{-(m-n)})^T.$$
Denote $x=\sum_{i=-(m-n)}^0(-a_iI_{i-1}-b_iJ_{i-1})$.
Then $\xi_x=\exp(ad_x): \GG\to \GG$ is a Lie algebra automorphism,
where the action of $ad_x$ on $\GG$ is the adjoint action.
Let $W_{\psi_{m,n}}^{\xi_x}$ be the new Whittaker module with action
$y\circ w_{\psi_{m,n}}=\xi_x(y)w_{\psi_{m,n}}$ for any $y\in\GG$,
which is equivalent to $W_{\psi_{m,n}}$ as $\GG$-modules.
By definition of $\xi_x$ and Eqs.~\eqref{derived-algebra}, \eqref{matrix},
we get 
$$L_p\circ w_{\psi_{m,n}}=H_p\circ w_{\psi_{m,n}}=0, \quad \forall p\geq m+n,$$ and the actions of $\widetilde{IJ}$ and ${\bf c_1, c_2, c_3}$ are unchanged.
Therefore, without loss of generality, we may assume that
\begin{equation}\psi_{m,n}(L_p)=\psi_{m,n}(H_p)=0,~\forall p\geq m+n.\label{twisted}\end{equation}
Then by Proposition \ref{m>n-0} and Eq.~\eqref{twisted} we deduce that the $\Ind_{\GG^{(m, n)}}^{\GG^{(m,0)}} \C w_{\psi_{m,n}}$ is an irreducible 
$\GG^{(m, 0)}$-module and satisfies the conditions of Proposition \ref{m<n-1} (with $k=m+n-1$).
So $\Ind_{\GG^{(m, 0)}}^{\GG_0}\Ind_{\GG^{(m, n)}}^{\GG^{(m,0)}} \C w_{\psi_{m,n}}$ is an irreducible $\GG_0$-module 
and  satisfies the conditions of Lemma \ref{CY-Lemma} (with $d=0, k=m+n-1$). Therefore $\Ind_{\GG_0}^{\GG}\Ind_{\GG^{(m, 0)}}^{\GG_0}\Ind_{\GG^{(m, n)}}^{\GG^{(m,0)}} \C w_{\psi_{m,n}}$ is an irreducible $\GG$-module, that is, $W_{\psi_{m,n}}$ is an irreducible $\GG$-module.
	
	$(\Rightarrow)$
	Assume that $\psi_{m,n}(I_{m+n-1})\psi_{m,n}(J_{m+n-1})=0$.
	Then $\psi_{m,n}(I_{m+n-1})=0$ or $\psi_{m,n}(J_{m+n-1})=0$. 
	If $\psi_{m,n}(I_{m+n-1})=0$, denote $w_1=I_{n-1}w_{\psi_{m,n}}$, then
	$y\cdot w_1=\psi_{m,n}(y)w_1$ for any $y\in\GG^{(m, n)}$. 
	By the PBW Theorem it is clear that $w_1$ generates a nonzero proper submodule of $W_{\psi_{m,n}}$
	which contradicts that $W_{\psi_{m,n}}$ is an irreducible $\GG$-module. 
	Similarly, if $\psi_{m,n}(J_{m+n-1})=0$, 
	then $J_{n-1}w_{\psi_{m,n}}$ generates a nonzero proper submodule of $W_{\psi_{m,n}}$ which also yields a contradiction.
	So $\psi_{m,n}(I_{m+n-1})\psi_{m,n}(J_{m+n-1})\ne 0$.
\end{proof}

In the remaining parts, we give some discussions about the Whittaker module $W_{\psi_{m,n}}$ when $m\in\Z_+,n\in\N$ and $m,n$ have the different parities.

\begin{example}
	Let $m\in\Z_+$. Suppose that $\psi_{m,m+1}:\GG^{(m, m+1)}\rightarrow \C$ is a Whittaker function and $W_{\psi_{m,m+1}}=\Ind_{\GG^{(m, m+1)}}^{\GG} \C w_{\psi_{m,m+1}}$
	is the corresponding Whittaker module over $\GG$. Then $W_{\psi_{m,m+1}}$ is a reducible $\GG$-module.
\end{example}
\begin{proof}
	By Eq.~\eqref{derived-algebra} we see that
	$$\psi_{m,m+1}(L_{2m+1+j})=\psi_{m,m+1}(H_{2m+j})=\psi_{m,m+1}(I_{2m+1+j})=\psi_{m,m+1}(J_{2m+1+j})=0,\,\,\forall j\in\N.$$
	Then we consider the following three cases.
	
	Case 1: $\psi_{m,m+1}(I_{2m})=0$.
	
	Then denote $v=I_{m}w_{\psi_{m,m+1}}$. It is easy to see that 
	$$y\cdot v=\psi_{m,m+1}(y)v, \ \ \ \forall y\in \GG^{(m, m+1)}.$$
	Thus $v$ generates a nonzero proper submodule of $W_{\psi_{m,m+1}}$ by the PBW Theorem.
	So $W_{\psi_{m,m+1}}$ is reducible.
	
	Case 2: $\psi_{m,m+1}(J_{2m})=0$.
	
	Then $J_{m}w_{\psi_{m,m+1}}$ generates a nonzero proper submodule of $W_{\psi_{m,m+1}}$ by the similar discussions to Case 1.
	So $W_{\psi_{m,m+1}}$ is reducible.
	
	Case 3: $\psi_{m,m+1}(I_{2m})\psi_{m,m+1}(J_{2m})\ne 0$.
	
	Then denote $v=I_mw_{\psi_{m,m+1}}+\frac{\psi_{m,m+1}(I_{2m})}{\psi_{m,m+1}(J_{2m})}J_mw_{\psi_{m,m+1}}$, 
	by the direct computations we can obtain that
	$$y\cdot v=\psi_{m,m+1}(y)v, \ \ \ \forall y\in \GG^{(m, m+1)}.$$
	Therefore $v$ generates a nonzero proper submodule of $W_{\psi_{m,m+1}}$ by the PBW Theorem.
	So $W_{\psi_{m,m+1}}$ is reducible.
\end{proof}

\begin{example}
	Let $m\in\Z_+$. Suppose that $\psi_{m,m-1}:\GG^{(m, m-1)}\rightarrow \C$ is a Whittaker function and $W_{\psi_{m,m-1}}=\Ind_{\GG^{(m, m-1)}}^{\GG} \C w_{\psi_{m,m-1}}$
	is the corresponding Whittaker module over $\GG$. Then $W_{\psi_{m,m-1}}$ is a reducible $\GG$-module.
\end{example}
\begin{proof}
By Eq.~\eqref{derived-algebra} we see that
	$$\psi_{m,m-1}(L_{2m+1+j})=\psi_{m,m-1}(H_{2m+j})=\psi_{m,m-1}(I_{2m-1+j})=\psi_{m,m-1}(J_{2m-1+j})=0,\,\,\forall  j\in\N.$$
Then we consider the following three cases.
	
	Case 1: $\psi_{m,m-1}(I_{2m-2})=0$.
	
	Then denote $w=I_{2m-3}w_{\psi_{m,m-1}}$. It is easy to see that 
	$$y\cdot w=\psi_{m,m-1}(y)w, \ \ \ \forall y\in \GG^{(m, m-1)}.$$
	Thus $w$ generates a nonzero proper submodule of $W_{\psi_{m,m-1}}$ by the PBW Theorem.
	So $W_{\psi_{m,m-1}}$ is reducible.
	
	Case 2: $\psi_{m,m-1}(J_{2m-2})=0$.
	
	Then $J_{2m-3}w_{\psi_{m,m-1}}$ generates a nonzero proper submodule of $W_{\psi_{m,m-1}}$ by the similar discussions to Case 1.
	So $W_{\psi_{m,m-1}}$ is reducible.
	
	Case 3: $\psi_{m,m-1}(I_{2m-2})\psi_{m,m-1}(J_{2m-2})\ne 0$.
	
	From the proof of Theorem \ref{main-1} we can assume that $$\psi_{m,m-1}(L_p)=\psi_{m,m-1}(H_p)=0,\ \ \ ~\forall p\geq 2m-1.$$
	Then denote $w=L_{m-1}w_{\psi_{m,m-1}}$, we can deduce that
	$$y\cdot w=\psi_{m,m-1}(y)w, \ \ \ \forall y\in \GG^{(m, m-1)}.$$
	Therefore $w$ generates a nonzero proper submodule of $W_{\psi_{m,m-1}}$ by the PBW Theorem.
	So $W_{\psi_{m,m-1}}$ is reducible.
\end{proof}

\begin{example}
	Let $\psi_{1,4}:\GG^{(1, 4)}\rightarrow \C$ be a Whittaker function and $W_{\psi_{1,4}}=\Ind_{\GG^{(1, 4)}}^{\GG} \C w_{\psi_{1,4}}$
be the corresponding Whittaker module over $\GG$. Then $W_{\psi_{1,4}}$ is a reducible $\GG$-module.	
\end{example}
\begin{proof}
	By Eq.~\eqref{derived-algebra} we see that
	$$\psi_{1,4}(L_{3+j})=\psi_{1,4}(H_{2+j})=\psi_{1,4}(I_{5+j})=\psi_{1,4}(J_{5+j})=0,\,\,\forall j\in\N.$$
	Then we consider the following three cases.
	
	Case 1: $\psi_{1,4}(I_4)=0$.
	
	Then denote $w=I_3w_{\psi_{1,4}}$. It is clear that 
	$$y\cdot w=\psi_{1,4}(y)w, \ \ \ \forall y\in \GG^{(1, 4)}.$$
	Thus $w$ generates a nonzero proper submodule of $W_{\psi_{1,4}}$ by the PBW Theorem.
	So $W_{\psi_{1,4}}$ is reducible.
	
	Case 2: $\psi_{1,4}(J_4)=0$.
	
	Then $J_3w_{\psi_{1,4}}$ generates a nonzero proper submodule of $W_{\psi_{1,4}}$ by the similar discussions to Case 1.
	So $W_{\psi_{1,4}}$ is reducible.
	
	Case 3: $\psi_{1,4}(I_4)\psi_{1,4}(J_4)\ne 0$.
	
	Then denote $\al=\psi_{1,4}(I_4),  \be=\psi_{1,4}(J_4)$ and $$w=(a_1I_2+a_2J_2+a_3I_3^2+a_4J_3^2+a_5J_3I_3)w_{\psi_{1,4}},$$
	where $a_1, a_2, a_3, a_4, a_5\in\C$.
	It is easy to see that 
	$$(L_{2+j}-\psi_{1,4}(L_{2+j}))w=(H_{3+j}-\psi_{1,4}(H_{3+j}))w=(I_{4+j}-\psi_{1,4}(I_{4+j}))w=(J_{4+j}-\psi_{1,4}(J_{4+j}))w=0,$$
	for all $j\in\N$.
    Furthermore, by the direct computations we obtain
	\begin{align*}
	(H_2-\psi_{1,4}(H_2))w&=(\al a_1-\be a_2)w, \\
	(H_1-\psi_{1,4}(H_1))w&=(a_1I_3-a_2J_3+2\al a_3I_3-2\be a_4J_3-\be a_5 I_3+\al a_5 J_3)w \\
	&=(a_1+2\al a_3-\be a_5)I_3w+(-a_2-2\be a_4+\al a_5)J_3w,  \\
	(L_1-\psi_{1,4}(L_1))w&=(a_1I_3+a_2J_3+4\al a_3I_3+4\be a_4J_3+2\be a_5 I_3+2\al a_5 J_3)w  \\
	&=(a_1+4\al a_3+2\be a_5)I_3w+(a_2+4\be a_4+2\al a_5)J_3w.
    \end{align*}
Note that the following matrix 
\[
\begin{pmatrix}
\al&-\be&0&0&0\\            
1&0&2\al&0&-\be\\
1&0&4\al&0&2\be\\
0&-1&0&-2\be&\al  \\
0&1&0&4\be&2\al
\end{pmatrix}
\]
is not invertible.	
So there exist $a_1,a_2, a_3, a_4, a_5\in\C$ (which are not all zero) such that 
	\[
	\begin{pmatrix}
	\al&-\be&0&0&0\\            
	1&0&2\al&0&-\be\\
	1&0&4\al&0&2\be\\
	0&-1&0&-2\be&\al  \\
	0&1&0&4\be&2\al
	\end{pmatrix}
	\begin{pmatrix}
	a_1\\            
	a_2\\
	a_3\\
	a_4  \\
	a_5
	\end{pmatrix}=
	\begin{pmatrix}
	0\\            
	0\\
	0\\
	0  \\
	0
	\end{pmatrix}.
	\]
Thus there exists nonzero element $$w=(a_1I_2+a_2J_2+a_3I_3^2+a_4J_3^2+a_5J_3I_3)w_{\psi_{1,4}}$$
 such that 	
 $$(L_{1+j}-\psi_{1,4}(L_{1+j}))w=(H_{1+j}-\psi_{1,4}(H_{1+j}))w=(I_{4+j}-\psi_{1,4}(I_{4+j}))w=(J_{4+j}-\psi_{1,4}(J_{4+j}))w=0,$$
 for all $j\in\N$, which imply that
	$w$ generates a nonzero proper submodule of $W_{\psi_{1,4}}$ by the PBW Theorem.
	So $W_{\psi_{1,4}}$ is reducible.
\end{proof}

From the above examples we give the following conjecture.

{\bf Conjecture:} Let $m\in\Z_+, n\in\N$ and $m,n$ have the different parities.
Suppose that $\psi_{m,n}:\GG^{(m, n)}\rightarrow \C$ is a Whittaker function. 
Then the corresponding Whittaker $\GG$-module $W_{\psi_{m,n}}=\Ind_{\GG^{(m,n)}}^{\GG} \C w_{\psi_{m,n}}$ is always reducible.

\section{Irreducible $\GG$-modules from tensor products}

In this section, we study the tensor products of $\UU(\hh)$-free modules of rank one and Whittaker modules over $\GG$ since they are the most popular two families of non-weight modules.
Note that restricted modules include Whittaker modules, so we consider directly the tensor products of $\UU(\hh)$-free modules of rank one and restricted modules over $\GG$.
 At the end of subsection \ref{universal central extension} we see that there are three families of $\UU(\hh)$-free modules of rank one over $\GG$, denoted by  $\Omega(\lambda,\eta,\sigma,0),
\Omega(\lambda,\eta,0,\sigma)$ and $\Omega(\lambda,\delta,0,0)$ respectively. 
Let $R$ be an irreducible restricted module over $\GG$. 
We mainly investigate the irreducibility of $\Omega(\lambda,\eta,\sigma,0)\otimes R$ and $\Omega (\lambda,\eta,0,\sigma)\otimes R$
since $\Omega(\lambda,\delta,0,0)$ is always reducible.
More precisely, we give the irreducible criteria for the tensor products $\Omega(\lambda,\eta,\sigma,0)\otimes R$ and $\Omega (\lambda,\eta,0,\sigma)\otimes R$ viewed as $\GG$-modules.
Moreover, the necessary and sufficient conditions for these tensor product modules to be isomorphic are determined.

\subsection{Irreducibility}

In this subsection, we determine the irreducibility of the tensor product modules.
\begin{theorem}\label{tensor product modules}
 Suppose $\lambda\in\C^*,\eta\in\C,\sigma\in\C[X]$ with $\sigma\ne 0$. Let $R$ be an irreducible restricted module over $\GG$. The following conclusions hold.
 \begin{enumerate}[$(1)$]
  \item $\Omega(\lambda,\eta,\sigma,0)\otimes R$ is an irreducible $\GG$-module  if and only if $\sigma$ is a nonzero constant.
 \item $\Omega(\lambda,\eta,0,\sigma)\otimes R$ is an irreducible $\GG$-module  if and only if $\sigma$ is a nonzero constant.
\end{enumerate}
\end{theorem}
\begin{proof}
$(1)$
$(\Rightarrow)$. By Lemma \ref{U(h)-irr}, it is clear.

$(\Leftarrow)$.
Suppose that $\sigma$ is a nonzero constant. Let $V$ be a nonzero submodule of $\Omega(\lambda,\eta,\sigma,0)\otimes R$. Then there exists nonzero $$v=\sum_{j=0}^q\sum_{i=0}^p(X^iY^j\otimes v_{ij})\in V$$ such that $v_{iq}\ne 0$ for some $0\leq i\leq p$,
where $p,q\in\N, v_{ij}\in R$. 

Note that for any $0\leq i\leq p, 0\leq j\leq q$, there exists $N_{ij}\in\Z_+$ such that $$L_mv_{ij}=H_mv_{ij}=I_mv_{ij}=J_mv_{ij}=0,\quad \forall m\geq N_{ij},$$
by definition of restricted modules.
Denote $N_v=\max\{N_{ij}~|~0\leq i\leq p, 0\leq j\leq q\}$. 
Thus for any $m\geq N_v$,
\begin{align}\label{eq-Hm}
\lambda^{-m}H_mv=\sum_{j=0}^q\sum_{i=0}^p\big(X^{i+1}(Y-m)^j\otimes v_{ij}\big)\in V.
\end{align}
The right hand side of Eq.~\eqref{eq-Hm} can be written in the form $\sum_{j=0}^qm^jv_j$,
where $v_j\in \Omega(\lambda,\eta,\sigma,0)\otimes R$ are independent of $m$.
Taking $m=N_v, N_v+1,\cdots,N_v+q$, it is clear that the coefficient matrix of $v_j$, where $0\leq j\leq q$, is a Vandermonde matrix.
So $$v_j\in V,\quad \forall\ 0\leq j\leq q.$$ 
In particular, $v_q=(-1)^q\sum_{i=0}^p\big(X^{i+1}\otimes v_{iq}\big)\in V$.
Hence there exists nonzero $$v'=\sum_{i=0}^tX^i\otimes v_i\in V,$$
where $t\in\N, v_i\in R$ for $0\leq i\leq t$ with $v_t\ne 0$.
Choose the such nonzero $v'=\sum_{i=0}^tX^i\otimes v_i\in V$ such that $t$ is minimal.
\begin{claim}\label{cla.1}
$t=0$.
\end{claim}
Assume $t>0$. For any $0\leq i\leq t$, there exists $N_i\in\Z_+$ such that $$L_mv_i=H_mv_i=I_mv_i=J_mv_{ij}=0, \quad \forall m\geq N_i.$$
Let $N_{v'}=\max\{N_i~|~0\leq i\leq t\}$. Taking $m\geq N_{v'}$,  we deduce
$$v'-\lambda^{-m}\sigma^{-1}I_mv'=\sum_{i=0}^t\big{(}X^i-(X-1)^i\big{)}\otimes v_i=tX^{t-1}\otimes v_t+\sum_{i=0}^{t-2}X^i\otimes v_i'\in V,$$
where $v_i'\in R$ for $0\leq i\leq t-2$, which contradicts that $t$ is minimal. So $t=0$. Claim \ref{cla.1} is proved.

From Claim \ref{cla.1} we see that there exists nonzero $w\in R$ such that $1\otimes w\in V$.
Furthermore, there exists $N_w\in\Z_+$ such that $$L_mw=H_mw=I_mw=J_mw=0,\quad \forall m\geq N_w.$$
\begin{claim}\label{cla.2}
$X^iY^j\otimes w\in V,\ \forall i,j\in\N$.
\end{claim}
We prove the claim by induction on $i+j$. First, $X^iY^j\otimes w\in V$ is trivial if $i+j=0$.
Suppose the results hold for $i+j\leq k$, where $k\in\N$. For $i+j=k+1$, it is easy to see that $i\geq 1$ or $j\geq 1$.
If $i\geq 1$, then $X^{i-1}Y^j\otimes w\in V$. For arbitrary $m, m'>N_w$ with $m\ne m'$, we compute
$$\lambda^{-m}L_m(X^{i-1}Y^j\otimes w)=X^{i-1}(Y-m)^j(Y-mX+m\eta)\otimes w\in V,$$
which implies
\begin{equation}\label{m}
X^{i-1}Y^j(Y-mX)\otimes w\in V,
\end{equation}
by inductive hypothesis.
Similarly, we can get
\begin{equation}\label{m'}
X^{i-1}Y^j(Y-m'X)\otimes w\in V.
\end{equation}
From Eqs. \eqref{m} and \eqref{m'}, we can obtain $X^iY^j\otimes w\in V$.
If $j\geq 1$, we also can get $X^iY^j\otimes w\in V$ by similar discussions. Claim \ref{cla.2} is proved.

From Claim \ref{cla.2} we see $\Omega(\lambda,\eta,\sigma,0)\otimes w\subseteq V$.
Denote $$W=\{w'\in R~|~\Omega(\lambda,\eta,\sigma,0)\otimes w'\subseteq V\}.$$  It is clear that $W$ is nonzero since $w\in W$.
For any $w'\in W$, we deduce
\begin{align*}
&\Omega(\lambda,\eta,\sigma,0)\otimes L_m(w')\subseteq L_m(\Omega(\lambda,\eta,\sigma,0)\otimes w')-L_m\Omega(\lambda,\eta,\sigma,0)\otimes w'\subseteq V,\\
&\Omega(\lambda,\eta,\sigma,0)\otimes H_m(w')\subseteq H_m(\Omega(\lambda,\eta,\sigma,0)\otimes w')-H_m\Omega(\lambda,\eta,\sigma,0)\otimes w'\subseteq V,\\
&\Omega(\lambda,\eta,\sigma,0)\otimes I_m(w')\subseteq I_m(\Omega(\lambda,\eta,\sigma,0)\otimes w')-I_m\Omega(\lambda,\eta,\sigma,0)\otimes w\subseteq V, \\
&\Omega(\lambda,\eta,\sigma,0)\otimes J_m(w')=J_m(\Omega(\lambda,\eta,\sigma,0)\otimes w')\subseteq V, \quad \forall m\in\Z,\\
&\Omega(\lambda,\eta,\sigma,0)\otimes {\bf c_i}(w')={\bf c_i}(\Omega(\lambda,\eta,\sigma,0)\otimes w')\subseteq V,\quad \forall i=1,2,3.
\end{align*}
Hence $$L_m(w')\in W,\ H_m(w')\in W, \ I_m(w')\in W,\ J_m(w')\in W, \ \forall m\in\Z, \quad {\bf c_i}(w')\in W,\  \forall i=1,2,3.$$
which show that $W$ is a nonzero submodule of $R$.  Thus $W=R$ since $R$ is irreducible. So $V=\Omega(\lambda,\eta,\sigma,0)\otimes R$.
This completes the proof of $(1)$.

$(2)$ It is similar to $(1)$.
\end{proof}

By Theorems \ref{main-1} and \ref{tensor product modules} the following corollary is trivial.
\begin{corollary}
Let $\lambda,\sigma\in\C^*,\eta\in\C, m\in\Z_+, n\in\N$ and $m,n$ have the same parity.
Suppose that $\psi_{m,n}:\GG^{(m, n)}\rightarrow \C$ is a Whittaker function
and $W_{\psi_{m,n}}=\Ind_{\GG^{(m, n)}}^{\GG} \C w_{\psi_{m,n}}$ is the corresponding Whittaker module over $\GG$.
Then the following results hold.
	\begin{enumerate}[$(1)$]
		\item $\Omega(\lambda,\eta,\sigma,0)\otimes W_{\psi_{m,n}}$ is an irreducible $\GG$-module if and only if $\psi_{m,n}(I_{m+n-1})\psi_{m,n}(J_{m+n-1})\ne 0$.
		\item $\Omega(\lambda,\eta,0,\sigma)\otimes W_{\psi_{m,n}}$ is an irreducible $\GG$-module if and only if $\psi_{m,n}(I_{m+n-1})\psi_{m,n}(J_{m+n-1})\ne 0$.
	\end{enumerate}
\end{corollary}

\subsection{Isomorphism}

In this subsection, the isomorphism classes of these tensor product modules are obtained.
\begin{theorem}\label{iso}
 Let $\lambda,\lambda_1,\sigma,\sigma_1\in\C^*,\eta,\eta_1\in\C$.  Suppose that $R$ and $S$ are irreducible restricted modules over $\GG$. Then the following statements hold.
 \begin{enumerate}[$(1)$]
  \item $\Omega(\lambda,\eta,\sigma,0)\otimes R\cong\Omega(\lambda_1,\eta_1,\sigma_1,0)\otimes S$
  if and only if $\lambda=\lambda_1, \eta=\eta_1, \sigma=\sigma_1$ and $R\cong S$.
 \item $\Omega(\lambda,\eta,0,\sigma)\otimes R\cong\Omega(\lambda_1,\eta_1,0,\sigma_1)\otimes S$
  if and only if $\lambda=\lambda_1, \eta=\eta_1, \sigma=\sigma_1$ and $R\cong S$.
  \item $\Omega(\lambda,\eta,\sigma,0)\otimes R$ and $\Omega(\lambda_1,\eta_1,0,\sigma_1)\otimes S$ 
   are not isomorphic.
\end{enumerate}
\end{theorem}
\begin{proof}
(1) The ``if part'' is trivial. We only need to show the ``only if'' part.
Let $$\varphi:\Omega(\lambda,\eta,\sigma,0)\otimes R\rightarrow\Omega(\lambda_1,\eta_1,\sigma_1,0)\otimes S$$ be a $\GG$-module isomorphism.
By definition of restricted modules, for any $w\in R$, there exists $N_w\in\Z_+$ such that
$$L_mw=H_mw=I_mw=J_mw=0,\ \  \forall m\geq N_w.$$
Let $w\in R$ with $w\ne 0$. We can write
$$\varphi(1\otimes w)=\sum_{j=0}^q\sum_{i=0}^p(X^iY^j\otimes w_{ij}),$$
where $p,q\in\N, w_{ij}\in S$ with $w_{iq}\ne 0$ for some $0\leq i\leq p$.
Denote $p'=\max\{i'~|~0\leq i'\leq p, w_{i'q}\ne 0\}$.
\begin{claim}\label{iso1}
$\lambda=\lambda_1$.
\end{claim}
Denote $N'_w=\max\{N_w, N_{w_{ij}}~|~0\leq i\leq p, 0\leq j\leq q\}$.
For any $m,n\geq N'_w$, we compute
\begin{align*}
&\lambda^{-m}H_m\big(\sum_{j=0}^q\sum_{i=0}^p(X^iY^j\otimes w_{ij})\big)=(\frac{\lambda_1}{\lambda})^m\sum_{j=0}^q\sum_{i=0}^p\big(X^{i+1}(Y-m)^j\otimes w_{ij}\big),\\
&\lambda^{-n}H_n\big(\sum_{j=0}^q\sum_{i=0}^p(X^iY^j\otimes w_{ij})\big)=(\frac{\lambda_1}{\lambda})^n\sum_{j=0}^q\sum_{i=0}^p\big(X^{i+1}(Y-n)^j\otimes w_{ij}\big).
\end{align*}
So
\begin{align}\label{q}
0=&\varphi\bigg((\lambda^{-m}H_m-\lambda^{-n}H_n)(1\otimes w)\bigg)
=(\lambda^{-m}H_m-\lambda^{-n}H_n)\varphi(1\otimes w)\notag\\
=&(\lambda^{-m}H_m-\lambda^{-n}H_n)(\sum_{j=0}^q\sum_{i=0}^p(X^iY^j\otimes w_{ij}))\\
=&\sum_{j=0}^q\sum_{i=0}^pX^{i+1}\bigg((\frac{\lambda_1}{\lambda})^m(Y-m)^j-(\frac{\lambda_1}{\lambda})^n(Y-n)^j\bigg)\otimes w_{ij},\notag
\end{align}
which implies $\big((\frac{\lambda_1}{\lambda})^m-(\frac{\lambda_1}{\lambda})^n\big)X^{p'+1}Y^q\otimes w_{p'q}=0$. So $\lambda=\lambda_1$. Claim \ref{iso1} is proved.
\begin{claim}\label{iso2}
	$q=0$.
\end{claim}
Assume that $q>0$. By Claim \ref{iso1} we can take $\lambda=\lambda_1$ in equation \eqref{q}, thus we get 
$$0=\sum_{j=0}^q\sum_{i=0}^pX^{i+1}\bigg((Y-m)^j-(Y-n)^j\bigg)\otimes w_{ij},$$
which implies that $q(n-m)X^{p'+1}Y^{q-1}\otimes w_{p'q}=0$. But it is impossible when $m\ne n$. Contradiction.
So $q=0$. Claim \ref{iso2} is proved.

\begin{claim}\label{iso3}
 $\sigma=\sigma_1$. 
\end{claim}
From Claim \ref{iso2}, we can write $$\varphi(1\otimes w)=\sum_{i=0}^pX^i\otimes w_i,$$ where $p\in\N, w_i\in S$ with $w_p\ne 0$.
Denote $N''_w=\max\{N_w, N_{w_i}~|~0\leq i\leq p\}$. For any $m, n\geq N''_w$, we compute
\begin{align}\label{p}
\lambda^m\sigma\sum_{i=0}^pX^i\otimes w_i=\lambda^m\sigma\varphi(1\otimes w)=&\varphi(\lambda^m\sigma\otimes w)=\varphi(I_m(1\otimes w))\\
=&I_m\varphi(1\otimes w)=I_m(\sum_{i=0}^pX^i\otimes w_i)=\sum_{i=0}^p\lambda^m\sigma_1(X-1)^i\otimes w_i,\notag
\end{align}
since $\lambda=\lambda_1$, which yields
\begin{align*}
\lambda^m\sigma X^p\otimes w_p=\lambda^m\sigma_1X^p\otimes w_p.
\end{align*}
So $\sigma=\sigma_1$.
Claim \ref{iso3} is proved.

\begin{claim}\label{iso4}
	$\eta=\eta_1$. 
\end{claim}
By Claim \ref{iso3} we can take $\sigma=\sigma_1$ in Eq. \eqref{p}, then we obtain
$$\sum_{i=0}^p\big(X^i-(X-1)^i\big)\otimes w_i=0,$$ which indicates $p=0$.
Thus we deduce that for any $w\in R$, there exists $w'\in S$ such that
\begin{align}\label{1Ov}
\varphi(1\otimes w)=1\otimes w'.
\end{align}
Denote $N'''_w=\max\{N_w, N_{w'}\}$.
For any $m,n\geq N'''_w$, we have
\begin{align*}
X\otimes w'=\lambda^{-m}H_m(1\otimes w')&=\varphi(\lambda^{-m}H_m(1\otimes w))=\varphi(X\otimes w),\\
n(Y-mX+m\eta_1)\otimes w'&=n\lambda^{-m}L_m(1\otimes w') \\
&=\varphi(n\lambda^{-m}L_m(1\otimes w))=n(\varphi(Y\otimes w)-m\varphi(X\otimes w)+m\eta \varphi(1\otimes w)),\\
m(Y-nX+n\eta_1)\otimes w'&=m\lambda^{-n}L_n(1\otimes w') \\
&=\varphi(m\lambda^{-n}L_n(1\otimes w))=m(\varphi(Y\otimes w)-n\varphi(X\otimes w)+n\eta \varphi(1\otimes w)),
\end{align*}
which imply 
\begin{align}\label{XOv}
\varphi(X\otimes w)=X\otimes w',\ \varphi(Y\otimes w)=Y\otimes w',\ \eta=\eta_1.
\end{align}
Claim \ref{iso4} is proved.

Now we define the linear map $$f: R\rightarrow S$$ such that $\varphi(1\otimes w)=1\otimes f(w)$ for any $w\in R$.
It is clear that $f$ is an injection since $\varphi$ is a $\GG$-module isomorphism.  
By Eqs. ~\eqref{1Ov} and \eqref{XOv} we deduce
\begin{align}\label{varphi}
\varphi(X\otimes w)=X\otimes f(w),\ \ 
\varphi(Y\otimes w)=Y\otimes f(w).
\end{align}
For any $m\in\Z, w\in R$,  using Eq. \eqref{varphi} we obtain that
\begin{align*}
\varphi\big(L_m(1\otimes w)\big)=&\varphi\big(\lambda^m(Y-mX+m\eta)\otimes w\big)+\varphi\big(1\otimes L_m(w)\big)\\
=&\lambda^mY\otimes f(w)-m\lambda^m(X-\eta)\otimes f(w)+1\otimes f\big(L_m(w)\big),\\
L_m\big(1\otimes f(w)\big)=&\lambda^m(Y-mX+m\eta)\otimes f(w)+1\otimes L_m\big(f(w)\big)\\
=&\lambda^mY\otimes f(w)-m\lambda^m(X-\eta)\otimes f(w)+1\otimes L_m\big(f(w)\big),\\
\varphi\big(H_m(1\otimes w)\big)=&\varphi\big(\lambda^mX\otimes w+1\otimes H_m(w)\big)
=\lambda^mX\otimes f(w)+1\otimes f\big(H_m(w)\big),\\
H_m\big(1\otimes f(w)\big)=&\lambda^mX\otimes f(w)+1\otimes H_m\big(f(w)\big),\\
\varphi\big(I_m(1\otimes w)\big)=&\varphi\big(\lambda^m\sigma\otimes w+1\otimes I_m(w)\big)
=\lambda^m\sigma\otimes f(w)+1\otimes f\big(I_m(w)\big),\\
I_m\big(1\otimes f(w)\big)=&\lambda^m\sigma\otimes f(w)+1\otimes I_m\big(f(w)\big),\\
\varphi\big(J_m(1\otimes w)\big)=&\varphi\big(1\otimes J_m(w)\big)
=1\otimes f\big(J_m(w)\big),\\
J_m\big(1\otimes f(w)\big)=&1\otimes J_m\big(f(w)\big),
\end{align*}
which imply
\begin{align}\label{LHIJm}
f\big(L_m(w)\big)=L_m\big(f(w)\big), f\big(H_m(w)\big)=H_m\big(f(w)\big), f\big(I_m(w)\big)=I_m\big(f(w)\big), f\big(J_m(w)\big)=J_m\big(f(w)\big).
\end{align}
Moreover, it is easy to see that
\begin{align}\label{zz}
f({\bf c_i}(w))={\bf c_i}(f(w)),\ \ \forall 1\leq i\leq 3, w\in R.
\end{align}
 By Eqs.~\eqref{LHIJm} and \eqref{zz}, we see that $f: R\rightarrow S$ is a $\GG$-module homomorphism.
Note that $f$ is injective and $S$ is an irreducible $\GG$-module. Therefore we deduce that $f:R\rightarrow S$ is a $\GG$-module isomorphism. So $R\cong S$.
This completes the proof of $(1)$.

$(2)$ is similar to $(1)$.	

$(3)$ Assume that $$\phi:\Omega(\lambda,\eta,\sigma,0)\otimes R\rightarrow\Omega(\lambda_1,\eta_1,0,\sigma_1)\otimes S$$ is a $\GG$-module isomorphism. Let $w'\in S$ with $w'\ne 0$. Then there exist $p,q\in\N, v_{ij}\in R$ such that
$$\phi(\sum_{j=0}^q\sum_{i=0}^pX^iY^j\otimes v_{ij})=1\otimes w'.$$
Denote $N=\max\{N_{v_{ij}}, N_{w'}~|~0\leq i\leq p, 0\leq j\leq q\}$. Then for any $m\geq N$, we get
$$0=\phi(J_m(\sum_{j=0}^q\sum_{i=0}^pX^iY^j\otimes v_{ij}))=J_m(\phi(\sum_{j=0}^q\sum_{i=0}^pX^iY^j\otimes v_{ij}))
=J_m(1\otimes w')=\lambda_1^m\sigma_1\otimes w',$$
which is impossible. So $\Omega(\lambda,\eta,\sigma,0)\otimes R$ and $\Omega(\lambda_1,\eta_1,0,\sigma_1)\otimes S$ 
are not isomorphic.
\end{proof}

From Theorem \ref{iso}, we have the following corollary which agrees with \cite[Theorem 4.17]{CGZ}.
\begin{corollary}
 Let $\lambda,\lambda_1,\sigma,\sigma_1\in\C^*,\eta,\eta_1\in\C$.  Then the following conclusions hold.
 \begin{enumerate}[$(1)$]
 \item $\Omega(\lambda,\eta,\sigma,0)\cong\Omega(\lambda_1,\eta_1,\sigma_1,0)$
  if and only if $\lambda=\lambda_1,\eta=\eta_1,\sigma=\sigma_1$.
 \item $\Omega(\lambda,\eta,0,\sigma)\cong\Omega(\lambda_1,\eta_1,0,\sigma_1)$
  if and only if $\lambda=\lambda_1,\eta=\eta_1,\sigma=\sigma_1$.
  \item $\Omega(\lambda,\eta,\sigma,0)$ and $\Omega(\lambda_1,\eta_1,0,\sigma_1)$ are not isomorphic.
\end{enumerate}
\end{corollary}

\section{Examples}

In  this section, we construct some concrete examples of irreducible $\GG$-modules.

\begin{example}{\rm (Irreducible $\GG$-modules from irreducible Virasoro modules.)}
	\newline
	Let $V$ be an irreducible restricted module over $\VV$.
It is easy to see that $V$ is an irreducible restricted module over $\GG$ by defining $\HH V=\widetilde{IJ} V={\bf c_2}V=0$, where $\HH=\mathrm{span}\{H_m,{\bf c_3} ~|~m\in\Z\}$ is the Heisenberg algebra.
	Denote it by $V^\GG$. Then $\Omega(\lambda,\eta,\sigma,0)\otimes V^\GG$ and $\Omega(\lambda,\eta,0,\sigma)\otimes V^\GG$ are irreducible $\GG$-modules by Theorem \ref{tensor product modules}, where $\lambda, \sigma\in\C^*, \eta\in\C$.
	Moreover, there are many examples of irreducible restricted $\VV$-modules in \cite{MZ1}. 
	So we could get many irreducible restricted $\GG$-modules by this way.
\end{example}

\begin{example}{\rm (Irreducible $\GG$-modules from irreducible Heisenberg-Virasoro modules.)}
	\newline
	Let $W$ be an irreducible restricted module over $\LL$.
	Then $W$ is an irreducible restricted module over $\GG$ by defining $\widetilde{IJ} V=0$.
	Denote it by $W^\GG$. Then $\Omega(\lambda,\eta,\sigma,0)\otimes W^\GG$ and $\Omega(\lambda,\eta,0,\sigma)\otimes W^\GG$ are irreducible $\GG$-modules by Theorem \ref{tensor product modules}, where $\lambda, \sigma\in\C^*, \eta\in\C$.
    Furthermore, there are many examples of irreducible restricted $\LL$-modules in \cite{CG, TYZ}. 
	Thus we could get irreducible tensor product modules over $\GG$.
\end{example}

\begin{example}
Let $m\in\Z_+, n\in\N$ and $m,n$ have the same parity. Then we can construct irreducible Whittaker $\GG$-modules $W_{\psi_{m,n}}$ by Theorem \ref{main-1}, which are all restricted modules. 
Then by Theorem \ref{tensor product modules} we could obtain irreducible tensor product modules $\Omega(\lambda,\eta,\sigma,0)\otimes W_{\psi_{m,n}}$ and $\Omega(\lambda,\eta,0,\sigma)\otimes W_{\psi_{m,n}}$ over $\GG$, where $\lambda, \sigma\in\C^*, \eta\in\C$.
Moreover, the authors construct some irreducible restricted $\GG$-modules in \cite{CY}. So we can get  irreducible tensor product modules over $\GG$.
\end{example}

In the remaining parts we construct irreducible $\GG$-modules from Block modules.
\begin{example}
Set $$X_0={\rm span}\{L_0, H_0, I_0, J_0, {\bf c_1, c_2, c_3}\}, \  X_0'={\rm span}\{H_0, I_0\}.$$
It is clear that $X_0$ and $X_0'$ are subalgebras of $\GG$ and the center of $X_0$ is $\C L_0+\C {\bf c_1}+\C {\bf c_2}+\C {\bf c_3}$.
Recall that all irreducible modules over $X_0'$ are obtained by Block in \cite{B}. 
Let $V$ be an infinite-dimensional irreducible module over $X_0'$. Then the action of $I_0$ on $V$ is bijective.
Defining 
$$\GG^+v=0,\ J_0v=\lambda I_0^{-1}v,\ L_0v=\mu v,\ {\bf c_j}v=c_j v,\ \forall 1\leq j\leq 3, v\in V,$$
where $\lambda\in\C^*, \mu, c_i\in\C$. Then $V$ is an irreducible $\GG_0$-module satisfying the conditions of Lemma \ref{CY-Lemma} (with $k=d=0$). So the induced module $\Ind_{\GG_0}^\GG(V)$ is an irreducible $\GG$-module. 
Moreover, we could get $\Omega(\lambda,\eta,0,\sigma)\otimes\Ind_{\GG_0}^\GG(V)$ and $\Omega(\lambda,\eta,\sigma,0)\otimes\Ind_{\GG_0}^\GG(V)$ are irreducible $\GG$-modules by Theorem \ref{tensor product modules}, where $\lambda, \sigma\in\C^*, \eta\in\C$.
\end{example}

\begin{example}For any $d\in 2\N+1$, set 
	 $$Y_d=\sum_{i\in\Z_+}(\C L_i+\C H_i+\C I_{i-d}+\C J_{i-d}).$$
	 It is clear that $Y_d$ is an ideal of $\GG_d$. Thus we have the quotient algebra $X_d=\GG_d/Y_d$.
We denote $$X_d={\rm span}\{L_0, H_0, I_{-d}, J_{-d}, {\bf c_1, c_2, c_3}\},\ X_d'={\rm span}\{H_0, I_{-d}, J_{-d}, {\bf c_1, c_2, c_3}\},  \ X_d''={\rm span}\{H_0, I_{-d}\}$$
by abusing notations.
	Let $V$ be an infinite-dimensional irreducible module over $X_d''$. Then the action of $I_{-d}$ on $V$ is bijective.
	Defining 
	$$J_{-d}v=\lambda I_{-d}^{-1}v,\ \ {\bf c_j}v=c_j v,\ \forall 1\leq j\leq 3, v\in V,$$
	where $\lambda\in\C^*, c_i\in\C$. Then $V$ is an irreducible $X_d'$-module.
	Denote $W=\Ind_{X_d'}^{X_d}(V)$, which is the induced module over $X_d$.
	\begin{lemma}
		$W$ is an irreducible $X_d$-module.
	\end{lemma}
	\begin{proof}
	For any nonzero $w\in W$, we could write $w$ in the form $$\sum_{r=0}^qL_0^rv_r$$	by the PBW Theorem, where $q\in\N, v_r\in V$ and $v_q\ne 0$. 
	Let $\<w\>$ denote the submodule of $W$ generated by $w$.
	If $q=0$, then it is easy to see that $\<w\>=W$.
	If $q>0$, we compute 
	\begin{align*}
	(I_{-d}J_{-d}-\lambda)(w)&=\sum_{r=0}^q(I_{-d}J_{-d}-\lambda)L_0^rv_r \\
	&=\sum_{r=0}^q[I_{-d}J_{-d}, L_0^r]v_r=2dq\lambda L_0^{q-1}v_q+{\rm \ lower-terms\ of \ } L_0.
	\end{align*}
Thus we also can deduce that $\<w\>=W$. So $W$ is an irreducible $X_d$-module.
	\end{proof}

By defining $Y_d W=0$, we could get that $W$ is an irreducible $\GG_d$-module  
	satisfying the conditions of Lemma \ref{CY-Lemma} (with $k=-d$). Therefore the induced module $\Ind_{\GG_d}^\GG(W)$ is an irreducible $\GG$-module. 
	Furthermore, we could obtain that $\Omega(\lambda,\eta,0,\sigma)\otimes\Ind_{\GG_d}^\GG(W)$ and $\Omega(\lambda,\eta,\sigma,0)\otimes\Ind_{\GG_d}^\GG(W)$ are irreducible $\GG$-modules by Theorem \ref{tensor product modules}, where $\lambda, \sigma\in\C^*, \eta\in\C$.
\end{example}

\noindent {\bf Acknowledgements.} This work is supported by NSFC
(11931009, 12401032), China National Postdoctoral
Program for Innovative Talents (BX20220158), China
Postdoctoral Science Foundation (2022M711708), Fundamental
Research Funds for the Central Universities and Nankai Zhide
Foundation.

\noindent {\bf Declaration of interests. } The author has no conflicts of interest to disclose.

\noindent {\bf Data availability. } Data sharing is not applicable
as no new data were created or analyzed.

\end{document}